\documentclass{article}
\usepackage{graphicx} 
\usepackage{amsmath, amsthm, amssymb}
\usepackage{fullpage}
\usepackage[many]{tcolorbox}
\usepackage{xcolor}
\usepackage{times}
\usepackage{wasysym}
\usepackage{mathtools}
\usepackage{mathrsfs}
\usepackage{enumitem}

\usepackage[pdftex,plainpages=false,hypertexnames=false,pdfpagelabels]{hyperref}
\definecolor{medium-blue}{rgb}{0,0,.8}
\hypersetup{colorlinks, linkcolor={purple}, citecolor={medium-blue}, urlcolor={medium-blue}}
\newcommand{\arxiv}[1]{\href{http://arxiv.org/abs/#1}{\tt arXiv:\nolinkurl{#1}}}
\newcommand{\arXiv}[1]{\href{http://arxiv.org/abs/#1}{\tt arXiv:\nolinkurl{#1}}}

\DeclareMathOperator{\FPdim}{FPdim}

\DeclareMathOperator{\Tr}{Tr}

\newcommand{\Rep}{\mathsf{Rep}}
\newcommand{\Fun}{\mathsf{Fun}}
\newcommand{\cZ}{\mathcal{Z}}
\newcommand{\unit}{\mathbf{1}}

\def\semicolon{;}
\def\applytolist#1{
    \expandafter\def\csname multi#1\endcsname##1{
        \def\multiack{##1}\ifx\multiack\semicolon
            \def\next{\relax}
        \else
            \csname #1\endcsname{##1}
            \def\next{\csname multi#1\endcsname}
        \fi
        \next}
    \csname multi#1\endcsname}

\def\calc#1{\expandafter\def\csname c#1\endcsname{{\mathcal #1}}}
\applytolist{calc}QWERTYUIOPLKJHGFDSAZXCVBNM;
\def\bbc#1{\expandafter\def\csname bb#1\endcsname{{\mathbb #1}}}
\applytolist{bbc}QWERTYUIOPLKJHGFDSAZXCVBNM;
\def\bfc#1{\expandafter\def\csname bf#1\endcsname{{\mathbf #1}}}
\applytolist{bfc}QWERTYUIOPLKJHGFDSAZXCVBNM;
\def\sfc#1{\expandafter\def\csname s#1\endcsname{{\sf #1}}}
\applytolist{sfc}QWERTYUIOPLKJHGFDSAZXCVBNM;
\def\fc#1{\expandafter\def\csname f#1\endcsname{{\mathfrak #1}}}
\applytolist{fc}QWERTYUIOPLKJHGFDSAZXCVBNM;
\def\rmc#1{\expandafter\def\csname rm#1\endcsname{{\mathrm #1}}}
\applytolist{rmc}QWERTYUIOPLKJHGFDSAZXCVBNM;

\numberwithin{equation}{section}

\theoremstyle{plain}
\newcounter{mytheorem}[section]

\newtheorem{thm}[mytheorem]{Theorem}
\newtheorem*{thm*}{Theorem}
\newtheorem{cor}[mytheorem]{Corollary}
\newtheorem{lem}[mytheorem]{Lemma}
\newtheorem{prop}[mytheorem]{Proposition}

\newtheorem*{claim*}{Claim}

\theoremstyle{definition}

\newtheorem*{trick*}{Trick}

\newtheorem{rem}[mytheorem]{Remark}

\newtheorem{conjecture}[mytheorem]{Conjecture}

\title{Classification of some $\bbZ/2\bbZ \times \bbZ/2\bbZ$-quadratic fusion categories \\of rank 6}
\date{}

\begin{document}
\author{Yue Meng and Zhiqiang Yu}
\maketitle
\begin{abstract}
A fusion category $\cC$ is said to be $\bbZ/2\bbZ \times \bbZ/2\bbZ$-quadratic if the group $G(\cC)$ of invertible objects is isomorphic  to $\bbZ/2\bbZ \times \bbZ/2\bbZ$, and the remaining simple objects form an orbit under the action of $G(\cC)$.
In this paper, we give a partial classification of spherical $\bbZ/2\bbZ \times \bbZ/2\bbZ$-quadratic fusion categories of rank six.
More precisely,  we show that  its Grothendieck ring $\cK_0(\cC)$ must be one of nine fusion rings if the fusion rule multiplicities are less than $20$, and  the categorifications of five of them are previously known. We prove that   one of the last four fusion rings can be realized as  de-equivariantization of a near-group fusion category of type $\bbZ/2\bbZ \times \bbZ/4\bbZ+8$.

\end{abstract}



\section{Introduction}

Throughout this paper, we   work over an algebraically closed field $\bbK$ of characteristic $0$.

A fusion category $\cC$ is a semisimple $\bbK$-linear finite abelian tensor category, and we denote by $\cO(\cC)$ the set of isomorphism classes of simple objects of $\cC$.  The cardinality of $\cO(\cC)$ is called the rank of $\cC$. Fusion categories, generally, tensor categories connected deeply with representation theory of Hopf algebras and quantum groups, conformal field theory, low-dimensional topology, to name a few. Interested readers are referred to \cite{bakalovkirillov,drinfeld2010braided,egno2015,etingof2005fusion} and the references
therein for more details. One of the central problems in the research of fusion categories is their classifications.

A fusion ring  $R$ is categorifiable if there exists a fusion category $\cC$ such that the Grothendieck  ring $\cK_0(\cC)=R$, in this case  $\cC$ is also called a categorification of $R$.  For any finite group $G$, the categorifications of the group ring $\bbZ[G]$ are precisely  the pointed fusion categories $\text{Vec}_G^\omega$, the category of $G$-graded finite dimensional vector spaces of $\bbK$, where $\omega\in Z^3(G,\bbK^*)$ is a $3$-cocycle. It is known that for a  fusion ring, there are
only finitely many categorifications  up to tensor equivalence \cite[Theorem 2.28]{etingof2005fusion}, this is the Ocneanu rigidity. However,  given an arbitrary fusion ring, the existence of categorification can be  incredibly difficult. For example, pivotal fusion categories  have only been classified up to rank $3$ \cite{ostrik2015pivotal}, the classification of pseudo-unitary non-self-dual fusion categories of rank $4$ is complete \cite{larson,MIP2025}.

One of the most studied classes of fusion categories is the $G$-quadratic  fusion category with $G$ being a finite group. A fusion category $\cC$ is $G$-quadratic if non-invertible simple objects form an orbit under action of  $G$  \cite{grossmanizumi}, where $G$ is the group of invertible objects of $\cC$. The near-group fusion categories \cite{siehler} is a special class of $G$-quadratic categories. More precisely, 
let $G$ be a finite group and  $k$ be a  non-negative integer, a near-group fusion category $\cC$ is  of type $G+k$ if $\cK_0(\cC)$ is determined by  following fusion relations\begin{align*}
gX =X=Xg,\quad X^2=\sum_{g\in G}g+ kX.
\end{align*} 
If $G$ is abelian and  $k=0$, it is Tambara-Yamagami categories \cite{tambarayamagami}. There is much literature  considered categorifications of near-group fusion rings, we refer the readers
to  \cite{izumi,schopieray2} and the references therein for details.  
Given a finite group $G$, the progress on the question of when a $G$-quadratic fusion ring is categorifiable is slow, one particular reason  is that these  fusion rings pass the necessary categorifiable conditions automatically, such as  $d$-number test and pseudo-unitary inequality, see \cite{ostrik2009formal,ostrik2015pivotal}. 
 In \cite{MIP2025}, the authors give a complete classification of unitary $\bbZ/2\bbZ$-quadratic fusion categories. The classification of $\bbZ/3\bbZ$-quadratic fusion categories of rank $4$, or a near-group fusion category of type $\bbZ/3\bbZ+k$, is solved by \cite{larson}.

In this paper, following \cite{larson,MIP2025}, we consider the categorification of $\bbZ/2\bbZ\times\bbZ/2\bbZ$-quadratic fusion categories of rank $6$. Notice that $\bbZ/2\bbZ\times\bbZ/2\bbZ$-quadratic fusion categories of rank $5$ are exactly near-group fusion categories of type $\bbZ/2\bbZ\times\bbZ/2\bbZ+k$, which have been classified in \cite{schopieray2}. Under some restrictions on the multiplicity of fusion rules, we show that there are at most $9$ categorifiable fusion rings. Moreover, five  of these fusion rings have known fusion categories associated with them, and we find a new fusion category which is the de-equivariantization of a near-group fusion category of type $\bbZ/2\bbZ\times\bbZ/4\bbZ+8$.  The main results of this paper show the following:
\begin{thm}
Let $\cK_0(\cC)$ be the Grothendieck ring of a spherical $\bbZ/2\bbZ\times\bbZ/2\bbZ$-quadratic fusion category $\cC$ of rank $6$. If the fusion rule multiplicities are less than $20$, then $\cK_0(\cC)$ is isomorphic to one of
\[
R(0,0),\ R(1,0),\ R(2,0),\ R(2,2),\ R(4,4),\ R(4,6),\ R'(0),\ R'(2),\ R'(4),
\]
where $R(m,n)\cong R(n,m)$. More explicitly, these rings are the following:
\begin{itemize}[leftmargin=*, itemsep=2pt, parsep=2pt]
\item the Grothendieck ring of  $\mathrm{Vec}_{\bbZ/2\bbZ}^\omega\boxtimes \cA$, where $\cA$ is a near-group fusion category of type $\bbZ/2\bbZ+k$ with $k=0,1,2$, see Proposition \ref{notzero}.
\item the Grothendieck ring of the fusion category $\cC_{4,8,4}$, whose Drinfeld center is braided tensor equivalent to  $ \cC(\mathbb Z/2\mathbb Z\times \mathbb Z/2\mathbb Z ,\eta)\boxtimes \cC(\mathfrak{sl}_4,8)^0_{\mathbb Z/4\mathbb Z}$, see \textnormal{\cite[Theorem 5.2]{edie-michell}}.
\item the fusion ring with multiplication determined by
\begin{equation*}
g^2=h^2=\unit,\quad gh=hg,\quad g  X = Y = X g,\quad
X ^2 = \unit+ h + m X + m Y =
Y ^2,
\end{equation*}
where $m=2,4$.
\item the Grothendieck ring of fusion category $\cB$, which is a faithful $\bbZ/2\bbZ$-extension of pointed fusion subcategory $\mathrm{Vec}^{\,\omega}_{\mathbb Z/2\mathbb Z\times \mathbb Z/2\mathbb Z}$,
see Lemma \ref{R(0)}.
\item the Grothendieck ring of the de-equivariantization of near-group fusion category of type $\bbZ/2\bbZ \times \bbZ/4\bbZ+8$ by $\Rep(\bbZ/2\bbZ)$, see Corollary \ref{nonselfdual-categorification}.
\item the fusion ring with multiplication determined by
\begin{equation*}
g^2= h^2 = \unit, \quad gh=hg,\quad g X = X^*= Xg,\quad
X X^* = \unit+ h + 4 X +4 X^*= X^* X.  
\end{equation*}
\end{itemize}
\end{thm}
Theorem \ref{thm:PseudounitaryBoundselfdual} gives the self-dual assertion for $m+n<38$. Under the hypothesis $m,n<20$, the only remaining boundary pair is $(19,19)$; it is excluded by $4\mid(m+n)$ in Remark \ref{rem:K2abBetaConditions}. The non-self-dual assertion follows from Theorem \ref{thm:PseudounitaryBoundnonselfdual}, since $m<20<22$.

\begin{rem}
Although we have placed bounds in the 
fusion rule multiplicities, it is possible that we have listed all possibilities.  Namely, we believe that the Grothendieck ring of any $\bbZ/2\bbZ\times\bbZ/2\bbZ$-quadratic fusion category of rank $6$ must be one of  the above $9$ fusion rings.  And in  a follow-up paper \cite{mengyu}, we will classify $\bbZ/4\bbZ$-quadratic fusion categories of rank $6$, preliminary results show that similar conclusions also hold.
\end{rem}
The paper is organized as follows. In Section \ref{section-prels}, we recall some most used results on fusion categories. In Section \ref{section-selfdual}, we  prove that if a rank $6$ self-dual $\bbZ/2\bbZ\times\bbZ/2\bbZ$-quadratic fusion ring \ref{selfdual}   is categorifiable and $m+n<38$, then $(m,n)$ must be equal to one of $(0,0),(1,0),(0,1),(2,0),(0,2),(2,2),(4,4),(4,6),(6,4)$ (Theorem \ref{thm:PseudounitaryBoundselfdual}). In Section \ref{section-nonself-dual}, we prove that if  a rank $6$ non-self-dual $\bbZ/2\bbZ\times\bbZ/2\bbZ$-quadratic fusion ring   \ref{notselfdual} is categorifiable and $m<22$, then $m=0,2 \;\text{or}\;4$ (Theorem \ref{thm:PseudounitaryBoundnonselfdual}). In Section \ref{section-categorification}, we give categorifications of fusion rings $R(m,n)$ (Proposition \ref{selfdual-categorification}) and $R'(m)$ (Lemma \ref{R(0)}, Corollary \ref{nonselfdual-categorification}), respectively. 

\section{Preliminaries}\label{section-prels}
In this section, we recall some basic notions and properties of  fusion categories, braided fusion categories. For further details,
readers are referred to \cite{drinfeld2010braided,egno2015,etingof2005fusion}.

\subsection{Fusion categories}

Let $\cC$ be a fusion category.  There is a unique ring  homomorphism  $\FPdim(-)$  from the Grothendieck ring $\cK_0(\cC)$ to $\bbR$ such that $\FPdim(X)\geq1$ is an algebraic integer by \cite[Theorem 8.6]{etingof2005fusion},  $\forall X\in\cO(\cC)$
, and $\FPdim(X)$ is called the Frobenius-Perron dimension of object $X$. The Frobenius-Perron dimension  of fusion category $\cC$ is defined by
\begin{equation*}
\FPdim(\cC):=\sum_{X\in\cO(\cC)}\FPdim(X)^2.
\end{equation*}
We  say a fusion category $\cC$ is weakly integral if $\FPdim(\cC)\in\bbZ$; $\cC$ is integral if $\FPdim(X)\in\bbZ$ for all simple objects $X\in\cC$. For objects $X,Y\in\cC$, we   use the notation
$
[X,Y]:=\dim_\bbK\mathrm{Hom}_\mathcal{C}(X,Y)
$
for brevity.

Let $\cC$ be a pivotal fusion category with a pivotal structure $j$, that is, $j$ is a natural isomorphism from identity tensor functor $\text{id}_\cC$ to the double dual tensor functor $(-)^{**}$.  Then the quantum dimension of an object $X$ is defined by the (categorical) trace of $\text{id}_X$, that is,
\begin{align*}
\dim_j(X)=\text{Tr}(\text{id}_X):=\text{ev}_{X^*}\circ(j_X\circ \text{id}_X)\otimes \text{id}_{X^*}\circ\text{coev}_X,
\end{align*}
here we suppress the associativity and   unit constraints  of $\cC$, $\text{ev}$ and $\text{coev}$ are the evaluation and coevaluation morphisms of $\cC$, respectively.  $\cC$ is a spherical  fusion category  if $\dim_j(X)=\dim_j(X^*)$ for all objects $X$ of $\cC$. Then we define the global (or categorical) dimension of the spherical fusion category $\cC$ as
\begin{equation*}
\dim(\cC):=\sum_{X\in\cO(\cC)}\dim_j(X)^2.
\end{equation*}
Notice that $\dim_j(-)$ induces a homomorphism from $\cK_0(\cC)$ to $\bbK$
\cite[Proposition 4.7.12]{egno2015}. A fusion category $\cC$ is said to be pseudo-unitary if $\text{dim}(\cC)=\FPdim(\cC)$. Any weakly integral fusion category is pseudo-unitary \cite{etingof2005fusion}, for example.

Let $\cC$ be a fusion category. Then the algebra $\cK_0(\cC)\otimes_\bbZ\bbK$ is semisimple \cite[Proposition 3.7.3]{egno2015}. Given an irreducible representation $\chi$ of $\cK_0(\cC)\otimes_\bbZ\bbK$,  the element
$
\alpha_\chi=\sum_{X\in\cO(\cC)}\text{tr}_\chi(X)X^*
$
is central, where $\text{tr}_\chi(-)$ is the ordinary trace function on the representation $\chi$. Moreover, $\chi'(\alpha_\chi)=0$ if $\chi\ncong\chi'$ and $f_\chi:=\chi(\alpha_\chi)$ is a positive algebraic integer \cite{Lusztig}, $f_\chi$ is called a formal codegree of  $\cC$ \cite{ostrik2009formal}. For example, $\FPdim(\cC)$ and $\dim(\cC)$ are formal codegrees of $\cC$ determined by the homomorphisms $\FPdim(-)$ and $\dim(-)$, respectively. Moreover, for any formal codegree $f_\chi$ of spherical fusion category $\cC$,  there exists object
$A_{\chi} \in \cO (\cZ(\cC))$
such that $f_\chi =\frac{\dim(\cC)}{\dim(A_\chi )}$  and $[\unit_{\cC},\cF(A_\chi)]=[\cI(\unit_{\cC}),A_\chi]=\dim(\chi)$ by \cite[Theorem 2.13]{ostrik2015pivotal}, where $\cZ(\cC)$ is  the Drinfeld center of  $\mathcal C$, $\cF:\cZ(\cC)\to\cC$ is the forgetful functor and $\cI:\cC\to\cZ(\cC)$ is the adjoint functor of $\cF$.

\subsection{Modular fusion category and modular data}
Let $\cC$ be a braided spherical fusion category with a braiding $c$, that is, $\cC$ is a pre-modular fusion category.  Then $\cC$ is modular  
if its $S$-matrix is non-degenerate \cite{egno2015}, where  $S=(S_{X,Y})$ and $S_{X,Y}=\operatorname{Tr}(c_{Y,X^*}c_{X^*,Y})$, $\forall X,Y\in\mathcal{O}(\mathcal{C})$. Equivalently,  $\mathcal C$ is modular if its M\"{u}ger center is trivial \cite{muger2003}.  
Let $T=(\delta_{X,Y}\theta_X)$, where $\theta_X$ is the scalar associated with the ribbon twist of $X\in\mathcal O(\mathcal C)$. 
The pair of $(S,T)$ is referred to as the \textit{modular data}.

It is well-known that the S-matrix of a modular fusion category $\cC$ determines the multiplication of the Grothendieck ring $\cK_0(\cC)$ by the famous Verlinde formula \cite{egno2015} which states that for any objects $X, Y, Z\in\cO(\cC)$,
\begin{equation}
N_{X, Y}^Z=\frac{1}{\operatorname{dim}(\mathcal{C})} \sum_{W \in \mathcal O (\mathcal{C})} \frac{S_{X, W} S_{Y, W} S_{Z^\ast, W}}{S_{\unit_{\cC}, W}}\label{verlinde},
\end{equation}
where $N_{X, Y}^Z:= [X \otimes Y, Z]$.

Let $\mathcal C$ be modular fusion 
 category with modular data $(S,T)$. 
Denote by $\operatorname{\mathbb{Q}}(S)$ and $\operatorname{\mathbb{Q}}(T)$ the smallest fields containing all entries of $S$ and $T$, respectively.
It is proved in \cite{ngschauenburgcongruence,donglinng} that $\mathbb{Q}(S)\subseteq \mathbb{Q}(T)$ and $\mathbb{Q}(T)=\mathbb{Q}_N:=\mathbb{Q}(\zeta_N)$, where $ \zeta_N $ is a primitive $N$-th root of unity and $N$ is the order of the $T$ matrix.  It follows from  
  \cite[Lemma 2.4]{muger2003} that
\begin{equation}
 \frac{S_{X,Y} S_{X,Z}}{S_{\unit_{\cC},X}^2}= \sum_{W\in \mathcal{O}(\mathcal C)}N_{Y,Z}^W\frac{S_{X,W}}{S_{\unit_{\cC},X}},
\end{equation}
hence the map $ \chi_Y:X\mapsto \frac{S_{X,Y}}{S_{\unit_{\cC},Y}} $ defines a linear character of the Grothendieck ring  $ \mathcal{K}_0(\mathcal C) $.
Since $ S$  is non-degenerate, $ \{\chi_Y\}_{Y\in\mathcal O(\mathcal C)} $  forms all characters of $ \mathcal{K}_0(\mathcal C) $. Let $ \operatorname{Gal}(\mathbb{Q}(S)/\mathbb{Q}) $ be the Galois group of field extension $\mathbb{Q}(S)/\mathbb{Q}$, then for any $ \sigma\in \operatorname{Gal}(\mathbb{Q}(S)/\mathbb{Q}) $,  $ \sigma(\chi_Y) $  is also a linear character of $ \mathcal{K}_0(\mathcal C) $.  Thus, there is a unique permutation $\hat{\sigma}:\mathcal{O}(\mathcal{C})\to\mathcal{O}(\mathcal{C})$ such that $\sigma(\chi_Y) =\chi_{\hat{\sigma}(Y)}$, i.e.,
\begin{equation}
\sigma\left(\dfrac{S_{X,Y}}{S_{\unit_{\cC},Y}}\right)=\dfrac{S_{X,\hat{\sigma}(Y)}}{S_{\unit_{\cC},\hat{\sigma}(Y)}}, \quad \forall  X, Y \in \mathcal{O}(\mathcal{C}). 
\end{equation}
 By using the Verlinde formula (\ref{verlinde}), we have
\begin{equation}\label{twentyseven}
\dim(\hat{\sigma}(X))^2=\dfrac{\dim(\mathcal{C})}{\sigma(\dim(\mathcal{C}))}\sigma(\dim(X))^2,\quad \forall X\in\cO(\cC).
\end{equation}

\subsection{ \texorpdfstring{$\bbZ/2\bbZ\times \bbZ/2\bbZ$}{Z/2Z}-quadratic fusion categories of rank  six }
We call a spherical fusion category $\cC$ a $\bbZ/2\bbZ\times \bbZ/2\bbZ$-quadratic fusion category of rank six if  $G(\cC) = \{\unit_{\cC},g,h,gh\}=\bbZ/2\bbZ\times\bbZ/2\bbZ$ and $\cC$ has  one other orbit of simple objects under the $G(\cC)$-action.
It is easy to see that  the fusion rules are  commutative, and there are two cases:
\begin{enumerate}[label=(Q\arabic*), labelindent=1cm, leftmargin=1.5cm]
\item
\label{Q:4ObjectsNonSelfDual} $\cO(\cC)=\{\unit_{\cC},g,h,gh,X,Y\}$; 
fusion rules determined by: $g\otimes X=Y$,
$X^2 = \unit_{\cC}\oplus h \oplus m X \oplus n Y$.
\item
\label{Q:4ObjectsSelfDual} $\cO(\cC)=\{\unit_{\cC},g,h,gh,X,X^*\}$; 
fusion rules determined by: $g\otimes X=X^*$,
$X^2 = g \oplus gh \oplus m X \oplus n X^*$.
\end{enumerate}
\noindent 
Notice that  $X\otimes X^*=(X\otimes X^*)^*$, so $m=n$ in the latter case.

By \cite[Lemma~3.1]{schopieray2}, $\cC$ is Galois conjugate
to a pseudo-unitary fusion category.
Since Galois conjugation preserves the fusion rules,
we may assume that $\cC$ is pseudo-unitary throughout
this paper.

We now compute the formal codegrees in the self-dual case.
Since $\cK_0(\cC)$ is commutative, its irreducible
representations are all one-dimensional.
Let $M$ be the matrix of left multiplication by the element
$\sum_{V\in\cO(\cC)}V^2$. Then
\begin{equation*}
M=
\begin{bmatrix}
6&0&2&0&2m&2n\\
0&6&0&2&2n&2m\\
2&0&6&0&2m&2n\\
0&2&0&6&2n&2m\\
2m&2n&2m&2n&8+2m^2+2n^2&4mn\\
2n&2m&2n&2m&4mn&8+2m^2+2n^2
\end{bmatrix}.
\end{equation*}
By \cite[Lemma~2.6]{ostrik2009formal}, the formal codegrees
of $\cC$ are the eigenvalues of $M$, counted with multiplicity:
\begin{align*}
f_1
&=8+(m+n)^2+(m+n)\sqrt{8+(m+n)^2},\quad
f_2=8+(m+n)^2-(m+n)\sqrt{8+(m+n)^2},\\
f_3
&=8+(m-n)^2+(m-n)\sqrt{8+(m-n)^2},\quad
f_4=8+(m-n)^2-(m-n)\sqrt{8+(m-n)^2},\\
f_5&=f_6=4.
\end{align*}
Let $d:=\FPdim(X)$. Since $\cC$ is pseudo-unitary, $
\dim(\cC)=\FPdim(\cC)=4+2d^2=f_1$, where 
\[
d:=\dim(X)=\dim(Y)=\frac{1}{2}\left(m+n+\sqrt{8+(m+n)^2}\right).
\]
In the non-self-dual case, with respect to the ordered basis
$\{\unit_{\cC},g,h,gh,X,X^*\}$, left multiplication by
$\sum_{V\in\cO(\cC)}VV^*$ is represented by $M$ with $n=m$,
so the formulas for the formal codegrees and $d$ remain valid
after this substitution.

\section{The Self-Dual Case}\label{section-selfdual}
Given a self-dual $\bbZ/2\bbZ \times \bbZ/2\bbZ$-quadratic  fusion category $\cC$ of rank six, let $\cO(\cC)=\{\unit_{\cC},g,h,gh,X,Y\}$  and fusion rules are determined by
\begin{align}
g^2=h^2=\unit_{\cC},\quad g \otimes X = Y = X \otimes g,\quad
X \otimes X &= \unit_{\cC}\oplus h \oplus m X \oplus n Y =
Y \otimes Y . \label{selfdual}\tag{$R(m,n)$}
\end{align}
Let us write \ref{selfdual} for such a fusion ring. 
\begin{lem}\label{lemfirst}
     Let $\cC$ be an integral fusion category such that $\cK_0(\cC)=$ \ref{selfdual}, then $m+n=1$. Moreover, $\cK_0(\cC)=\cK_0(\mathrm{Vec}^{\,\omega}_{\mathbb Z/2\mathbb Z} \boxtimes \Rep(S_3))$.
 \end{lem}
\begin{proof}
Let $t:=\sqrt{8+(m+n)^2}\in\bbZ$, equivalently, 
$(t-(m+n))(t+(m+n))=8$.
A direct computation shows $t\in\bbZ$  if and only if $t=3$, that is, $m+n=1$. In this case, the simple object $X$ or $Y$ generates a fusion subcategory which has the same fusion rules as $\Rep(S_3)$, this completes the proof.
\end{proof}

\begin{prop}\label{notzero}
If either $m$ or $n$ is zero, then  $(m,n)$ must be one of $(0,0),(0,1),(1,0),(0,2),(2,0)$. Therefore, $\cK_0(\cC)=\cK_0(\mathrm{Vec}^{\,\omega}_{\mathbb Z/2\mathbb Z} \boxtimes \cA)$, where $\cA$ is a near-group fusion category of type $\bbZ/2\bbZ+l$ with $l=0,1,2$.
\end{prop}
\begin{proof}
If $m=0$, then the fusion rule \ref{selfdual} implies that $\cC$ contains a near-group fusion subcategory  of type $\bbZ/2\bbZ+n$, which is generated by the simple object $Y$. It follows from \cite[Theorem 1.1]{ostrik2015pivotal} that $n\in\{0,1,2\}$. The argument for $n=0$ is same. 
\end{proof}

In the following, we always assume 
$m\neq0$ and
$n\neq 0$.

\subsection{Computation of the bi-adjoint functor}\label{subsection-decomposition}
In this subsection, we compute the bi-adjoint pair $(\cF,\cI)$:
$\cZ(\cC) \underset{\cI}{\overset{\cF}{\rightleftarrows}} \cC$. For any fusion category $\cC$, 
recall that 
\begin{align}
\cF(\cI(Y)) = \bigoplus_{X\in \cO(\cC)} X\otimes Y\otimes X^*, \; \forall \; Y\in\cO(\cC). \label{FIY}
\end{align}
For all $X\in\mathcal{C}$ and $Z\in\mathcal{Z}(\mathcal{C})$, the adjointness condition implies 
$
[\cI(X),Z]=[X,\cF(Z)]$. The following lemma is a direct result of 
\cite[Theorem 2.13]{ostrik2015pivotal}.
\begin{lem}\label{lem:Ostrik2.13}
Denote by
$
r := \sqrt{\frac{8 + (m+n)^2}{8 + (m-n)^2}}
$.  There are distinct simple objects 
$ \unit_{\mathcal{Z}(\cC)},A_1,A_2,A_3,A_4,A_5 \in \cO(\cZ(\cC))$
such that 
\begin{equation*}
\cI(\unit_\cC) = \unit_{\mathcal{Z}(\cC)} \oplus A_1\oplus A_2\oplus A_3\oplus A_4\oplus A_5,
\end{equation*}
\vspace{-0.5cm}
\begin{align}
\dim(A_1)&=\frac{f_1}{f_2} = 1 + \frac{(m+n)}{2}d = \frac{\dim(\cC)}{4}-1,\label{fpdima}\\
\dim(A_2) &=\frac{f_1}{f_3} 
=1 + \frac{1}{4}(m+n)d - \frac{r}{4}(m-n)d,\label{fpdimb}
\\
\dim(A_3) &=\frac{f_1}{f_4} 
=1 + \frac{1}{4}(m+n)d + \frac{r}{4}(m-n)d,\label{fpdimc}
\\
\dim(A_4) &=\frac{f_1}{f_5}=2 + \frac{(m+n)}{2}d,\label{fpdimd}
\\
\dim(A_5) &=\frac{f_1}{f_6}=2 + \frac{(m+n)}{2}d.\label{fpdime}
\end{align}
\end{lem}

\begin{rem}
\label{rem:K2abBetaConditions}
Since $\dim(A_j)=\dim(\cF(A_j))$ and the simple objects of $\cC$ have dimensions $1,d$, we have $\dim(A_j)\in\bbZ+\bbZ d$ for $1\leq j\leq5$. As $m\neq 0,n\neq0$, there can be only two possibilities:
\begin{itemize}
    \item $m=n$ and $4 \mid (m+n)$, or
    \item $m \neq n$, $r \in \mathbb{Q}(d)$, and $2 \mid (m+n)$.
\end{itemize}
Assume $ m \neq n$ below. Let 
$r = k_1 + k_2\sqrt{8+(m+n)^2}$,
where $k_1, k_2 \in \mathbb{Q}$.
Then 
\begin{equation*}
r^2 = \frac{8+(m+n)^2}{8+(m-n)^2} = k_1^2 + k_2^2(8+(m+n)^2) + 2k_1k_2\sqrt{8+(m+n)^2},
\end{equation*}
which means
$k_1k_2 = 0$ as $\sqrt{8+(m+n)^2} \notin \mathbb{Q}$. If $k_1=0$,  then $\frac{1}{8+(m-n)^2}=k_2^2$, so $|m-n|=1$, which contradicts $2\mid (m+n)$. Hence $k_2=0$ and $r \in \mathbb{Q}$. Write
$r=\frac{p}{q}$,  where \(p,q\in\mathbb Z_{>0}\) and \(\gcd(p,q)=1\).

Let
$a:=m+n, b:=m-n$, so $r=\sqrt{\frac{a^2+8}{b^2+8}}$. Since 
$\frac14(a-rb)\in \mathbb Z$ by Equation \eqref{fpdimb}, there exists \(t\in \mathbb Z\) such that $a-rb=4t$. Then we have
\begin{equation}
 qa-pb=4qt,
\label{r1}   
\end{equation}
we obtain
$pb\equiv0\pmod q$, which means $q\mid b$ as $\gcd(p,q)=1$. 
Assume $b=qs$ for some \(s\in\mathbb Z\), then $r^2=\frac{a^2+8}{b^2+8}=\frac{a^2+8}{q^2s^2+8}$, that is, 
$p^2(q^2s^2+8)=q^2(a^2+8)$, hence 
$8p^2\equiv0\pmod{q^2}$, so $q^2\mid 8$.
Therefore $q=1$  or   $q=2$. 
If \(q=2\), then \(p\) is odd and $ p^2(s^2+2)=a^2+8$. 
Meanwhile, Equation \eqref{r1}   means 
$ a-ps=4t$, so we have
$p^2=4pst+8t^2+4$, which implies \(2\mid p\), it is impossible as $p$ is odd. 
Therefore, \(q=1\) and 
$r=p\in\mathbb Z$.

It is straightforward to verify that if $0 \neq m \neq n \neq 0$, $r \in \bbZ$ and $2\mid (m+n)$, then $m+n \geq 10$. Hence, when $m+n<10$,  we have $m=n$ and $4\mid (m+n)$. Moreover, if  \( m + n = 10 \),  the only possible pairs \((m, n)\) are \((4, 6)\) and \((6, 4)\); if $m+n=38$, the only possible pairs \((m, n)\) are \((18, 20)\) and \((20, 18)\). The pairs \((18, 20)\) and \((20, 18)\) are the second smallest pair satisfying  $0 \neq m \neq n \neq 0$, $r \in \bbZ$ and $2\mid (m+n)$.
\end{rem}

\begin{prop}\label{biao}
Assume that $\cI(g)$, $\cI(h)$ and $\cI(gh)$ are multiplicity-free. The Drinfeld center $\mathcal{Z}(\cC)$ has  distinct simple objects 
$\unit_{\mathcal{Z}(\cC)},A_1,A_2,A_3,\dots$,
such that
\begin{align}
    \cI(\unit_\cC) &= \unit_{\mathcal{Z}(\cC)} \oplus A_1\oplus A_2\oplus A_3\oplus A_4\oplus A_5, 
        &&\text{and} & 
    \cI(g) &= A_6\oplus A_7\oplus A_8\oplus A_9\oplus A_{10}\oplus A_{11},\\
    \cI(h) &= A_4 \oplus A_5\oplus A_{12}\oplus A_{13}\oplus A_{14}\oplus A_{15}, 
        &&\text{and} & 
    \cI(gh) &= A_{10}\oplus A_{11}\oplus A_{16}\oplus A_{17}\oplus A_{18}\oplus A_{19}.
\end{align}
Denote the rest of the simple objects of $\mathcal{Z}(\cC)$ by $\{Z_s\}_{s\in \cS}$ where $\cS$ is a finite set.
The matrix $W$ of the forgetful functor $\cF:\mathcal{Z}(\cC) \to \cC$ can then be represented as follows, where zero entries are omitted:
\[
\setlength{\arraycolsep}{3.5pt}
\begin{array}{c|cccccc|cccccc|cccc|cccc|c}
& \unit_{\mathcal{Z}(\cC)} & A_1 & A_2 & A_3 & A_4 & A_5 & A_6 & A_7 & A_8 & A_9 & A_{10} & A_{11} & A_{12} & A_{13} & A_{14} & A_{15} & A_{16} & A_{17} & A_{18} & A_{19} & Z_s
\\\hline
\unit_\cC & 1 & 1 & 1 & 1 & 1 & 1 &&&&& &&&&& &&&&&
\\
g & &&&& && 1 & 1 & 1 & 1 &1 &1 &&&& &&&&
\\
h &&&&&1&1&&&& &&&1 &1 &1 &1 &&&&
\\
gh &&&& & &&&&&&1 &1&&&& &1 &1 &1 &1
\\\hline
X &&x_1 &x_2 &x_3 &x_4 &x_5 &x_6 &x_7 &x_8 &x_9 &x_{10} &x_{11} &x_{12} &x_{13} &x_{14} &x_{15} &x_{16} &x_{17} &x_{18} &x_{19} &z_s
\\
Y &&y_1 &y_2 &y_3 &y_4 &y_5 &y_6 &y_7 &y_8 &y_9 &y_{10} &y_{11} &y_{12} &y_{13} &y_{14} &y_{15} &y_{16} &y_{17} &y_{18} &y_{19}&z'_s
\end{array}
\vspace*{-\smallskipamount}
\]
and the induction matrix is given by $W^T$.
Moreover, the following equalities are satisfied:
\begin{align}
x_1 + y_1=x_4+y_4=x_5+y_5 &= \frac{m+n}{2},
\label{eqn:HorizontalSumsAB4}
\\
x_2 + y_2= \frac{1}{4}(m+n) - \frac{r}{4}(m-n), \quad
x_3 + y_3 &= \frac{1}{4}(m+n) + \frac{r}{4}(m-n),
\label{eqn:HorizontalSumsAB3}
\end{align}
\vspace{-0.5cm}
\begin{alignat}{2}
\sum_{i=1}^5x_i&= 2m \quad &&\text{ and }  \quad \sum_{i=1}^5y_i = 2n,\label{eqn:VerticalSumsAB1}
\\
\sum_{i=6}^{11}x_i &= 2n \quad &&\text{ and }  \quad \sum_{i=6}^{11}y_i= 2m,
\label{eqn:VerticalSumsAB2}
\\
\sum_{i=4}^5x_i+\sum_{i=12}^{15}x_i &= 2m \quad &&\text{ and }  \quad \sum_{i=4}^5y_i+\sum_{i=12}^{15}y_i = 2n,\label{eqn:VerticalSumsAB3}
\\
\sum_{i=10}^{11}x_i+\sum_{i=16}^{19}x_i&= 2n \quad &&\text{ and }  \quad \sum_{i=10}^{11}y_i+\sum_{i=16}^{19}y_i = 2m.\label{eqn:VerticalSumsAB4}
\end{alignat}
\end{prop}
\begin{proof}
 The decomposition of $\cI(\unit_{\cC}) $ is obtained in Lemma \ref{lem:Ostrik2.13}.  
Next, by Equation \eqref{FIY}, we have
\begin{align*}
\cF(\cI(\unit_{\cC}))&=6\cdot\unit_{\cC}\oplus2h\oplus2mX\oplus2nY, \quad
\cF(\cI(g))=6g\oplus2gh\oplus2nX\oplus2mY, \label{fi1} \\
\cF(\cI(h))&=2\cdot\unit_{\cC}\oplus6h\oplus2mX\oplus2nY, \quad
\cF(\cI(gh))=2g\oplus6gh\oplus2nX\oplus2mY.
\end{align*}
Equations \eqref{eqn:HorizontalSumsAB4} and \eqref{eqn:HorizontalSumsAB3} are consequences of Equations \eqref{fpdima}, \eqref{fpdimb}, \eqref{fpdimc}, \eqref{fpdimd}, and \eqref{fpdime}. Equation \eqref{eqn:VerticalSumsAB1} follows from  $[\cI(\unit_{\cC}),\cI(X)]=[\cF(\cI(\unit_{\cC})),X]=2m$ and $[\cI(\unit_{\cC}),\cI(Y)]=[\cF(\cI(\unit_{\cC})),Y]=2n$. Combining $[\cI(g),\cI(g)]=[\cF(\cI(g)),g]=6$ and $[\cI(g),\cI(\unit_{\cC})]=[\cF(\cI(g)),\unit_{\cC}]=0$, we obtain
\begin{equation*}
\cI(g) = A_6\oplus A_7\oplus A_8\oplus A_9\oplus A_{10}\oplus A_{11}.
\end{equation*}
Equation \eqref{eqn:VerticalSumsAB2} is then verified by $[\cI(g),\cI(X)]=[\cF(\cI(g)),X]=2n$ and $[\cI(g),\cI(Y)]=[\cF(\cI(g)),Y]=2m$.
Similarly, the identities $[\cI(h),\cI(h)]=[\cF(\cI(h)),h]=6$, $[\cI(h),\cI(g)]=[\cF(\cI(h)),g]=0$, and $[\cI(h),\cI(\unit_{\cC})]=[\cF(\cI(h)),\unit_{\cC}]=2$ imply
\begin{equation*}
\cI(h) = A_4 \oplus A_5\oplus A_{12}\oplus A_{13}\oplus A_{14}\oplus A_{15}.
\end{equation*}
Equation \eqref{eqn:VerticalSumsAB3} thus follows from $[\cI(h),\cI(X)]=[\cF(\cI(h)),X]=2m$ and $[\cI(h),\cI(Y)]=[\cF(\cI(h)),Y]=2n$.
Finally, from $[\cI(gh),\cI(gh)]=6$, $[\cI(gh),\cI(g)]=2$, $[\cI(gh),\cI(h)]=0$, and $[\cI(gh),\cI(\unit_{\cC})]=0$, we deduce
\begin{equation*}
\cI(gh) = A_{10}\oplus A_{11}\oplus A_{16}\oplus A_{17}\oplus A_{18}\oplus A_{19}.
\end{equation*}
Equation \eqref{eqn:VerticalSumsAB4} consequently holds by virtue of $[\cI(gh),\cI(X)]=2n$ and $[\cI(gh),\cI(Y)]=2m$.
\end{proof}

\begin{rem}
Until Remark \ref{all-decompositions}, we work with the multiplicity-free decomposition in Proposition \ref{biao}. Remark \ref{all-decompositions} treats all remaining decompositions, with their multiplicities retained, and establishes the same upper bound and twist restrictions.
\end{rem}

In what follows, we calculate the dimensions of all hom-spaces between $\cI(X)$ and $\cI(Y)$ by two different approaches. The first  is by taking adjoints and using Equation (\ref{FIY}). The second uses the induction matrix $W^T$ computed in Proposition \ref{biao}. This yields the
following equations:
\begin{align}
[\cI(X),\cI(X)] &= 8 + 2m^2 + 2n^2= \sum_{i=1}^{19}x_i^2 +\sum\nolimits_{s} z_s^2,\label{xx}\\
[\cI(X),\cI(Y)]&= 4mn= \sum_{i=1}^{19}x_iy_i +\sum\nolimits_{s} z_sz_s',\label{xy}\\
[\cI(Y),\cI(Y)]&= 8 + 2m^2 + 2n^2=  \sum_{i=1}^{19}y_i^2 +\sum\nolimits_{s} (z'_s)^{2}.\label{yy}
\end{align}

\begin{lem}
The non-negative integers $x_i,y_i$ for $6\leq i \leq 19$ and $z_s,z_s'$ for $s\in \cS$ satisfy the equation:
\begin{equation}
16+\frac{25}{8}(m +n)^2 - \frac{r^2}{8}(m-n)^2 = \sum_{i=6}^{19}(x_i+y_i)^2 +\sum\nolimits_{s} (z_s+z_s')^2. \label{mainequal}
\end{equation}
\end{lem}
\begin{proof}
First,  summing  Equations \eqref{xx} and \eqref{yy} and twice Equation \eqref{xy} together, then Equation (\ref{mainequal}) is a direct result of Equations \eqref{eqn:HorizontalSumsAB4} and \eqref{eqn:HorizontalSumsAB3}.
\end{proof}

\begin{prop}\label{prop:TwoThetasAB}
Denote the twists of $A_i$ by $\theta_i$, $1\leq i\leq 19$. 
Without loss of generality, the twists satisfy one of the following
\[
\begin{array}{ccccc|cccccc|cccc|cccc}
  \theta_1 & \theta_2 & \theta_3 & \theta_4 & \theta_5 & \theta_6 & \theta_7 & \theta_8 & \theta_9 & \theta_{10} & \theta_{11} & \theta_{12} & \theta_{13} & \theta_{14} & \theta_{15} & \theta_{16} & \theta_{17} & \theta_{18} & \theta_{19} \\
\hline
  1 & 1 & 1 & 1 & 1 & \pm1 &  \pm1 & \pm1 &  \pm1 & \mp1 & \mp1 & -1 & -1 & -1 & -1 & \pm1 & \pm1 & \pm1 & \pm1\\
 \hline
  1 & 1 & 1 & 1 & 1 & \pm1 &  \mp1 & \pm1 &  \mp1 & \pm1 & \mp 1 & -1 & -1 & -1 & -1 & \pm1 & \mp 1 & \pm1 & \mp 1\\
\hline
  1 & 1 & 1 & 1 & 1 & \pm i & \pm i & \pm i & \pm i & \mp i & \mp i & -1 & -1 & -1 & -1 &\pm i &\pm i &\pm i &\pm i\\
\hline
  1 & 1 & 1 & 1 & 1 & \pm i &  \mp i & \pm i &  \mp i & \pm i & \mp i & -1 & -1 & -1 & -1 & \pm i & \mp i & \pm i & \mp i
\end{array}.
\]
\end{prop}
\begin{proof}
By \cite[Theorem 2.5]{ostrik2015pivotal}, we know $\theta_i=1$ for $1 \leq i \leq 5$.
 Note   $\theta_4=\theta_5=1$ and $\nu_1(h)=\Tr_{\cZ(\cC)}(\theta_{\cI(h)})=0$ by \cite[Theorem 4.1]{ngschauenburg}, we have $\theta_j=-1$, $12\leq j\leq 15$. Then we calculate the  Frobenius-Schur indicators  of $g$. Specifically, 
\begin{align}
0 &=\nu_1(g)\dim(\cC)= \Tr_{\cZ(\cC)}(\theta_{\cI(g)})
    =\sum_{i=6}^{11} \theta_i \dim(A_i) 
   = \sum_{i=6}^9\theta_i+ 2\theta_{10} +2\theta_{11}+\Bigl(\sum_{i=6}^{11}\theta_i(x_i+y_i)\Bigr)d,\nonumber \\[1pt]
\pm \dim(\cC)&=\nu_2(g)\dim(\cC)=\Tr_{\cZ(\cC)}(\theta_{\cI(g)}^2)
     = \sum_{i=6}^{11} \theta_i^2 \dim(A_i) 
 = \sum_{i=6}^9\theta_i^2 + 2\theta_{10}^2 + 2\theta_{11}^2 +\Bigl(\sum_{i=6}^{11}\theta_i^2(x_i+y_i)\Bigr)d. \nonumber   
\end{align}
Since $g^2=\unit_{\cC}$, $|\nu_2(g)|=1$. The second indicator formula and $\sum_{i=6}^{11}\dim(A_i)=\dim(\cC)$ attain equality in the triangle inequality, so all $\theta_j^2$, $6\leq j\leq11$, are equal to $\nu_2(g)$. Together with Equation \eqref{eqn:VerticalSumsAB2}, this yields
\begin{equation*}
\pm\big(8+2(m+n)d\big)
= 8\theta_6^2 + \theta_6^2\Bigl(\sum_{i=6}^{11}(x_i+y_i)\Bigr)d= \theta_6^2\big(8+2(m+n)d\big).
\end{equation*}
Canceling the nonzero term $8+2(m+n)d$ gives $\theta_6^2=\pm1$, hence $\theta_j\in\{\pm1,\pm i\}$ for $6 \leq j \leq 11$. Since $\bbQ(i)\cap\mathbb Q(d)=\mathbb Q$ and $d\notin\mathbb Q$, the numbers $1,d$ are linearly independent over $\mathbb Q(i)$. Thus the first indicator formula gives $\theta_6+ \theta_7+ \theta_8+\theta_9+ 2\theta_{10}+ 2\theta_{11}=0$. By using this equation, we find that, up to permutation equivalence, there are only eight possibilities, as listed in the table. The analysis for the rest ribbons is similar to the above by considering  indicators $\nu_1(gh)$ and $\nu_2(gh)$.
\end{proof}

\begin{rem}
By Proposition \ref{prop:TwoThetasAB}, we know $\nu_2(\unit_{\cC})=\nu_2(h)=1$, and $\nu_2(g)=\nu_2(gh)=\pm 1$. Notice that the first two rows correspond to $\nu_2(g)=\nu_2(gh)=1$, and the last two rows correspond to $\nu_2(g)=\nu_2(gh)=-1$.
\end{rem}

\begin{lem}\label{m+n}
The constant term in the expression for $\dim(A_i)$
in the basis $\{1,d\}$ is $2$ for $i=4,5,10,11$. And we have the following equality:
$
(x_{4}+y_{4}) + (x_{5}+y_{5})=(x_{10}+y_{10}) + (x_{11}+y_{11}) = m+n.
$
\end{lem}
\begin{proof}
Let \(\sigma\) be the nontrivial element of \(\operatorname{Gal}(\mathbb{Q}(d)/\mathbb{Q})\), and let \(\widehat{\sigma}\) be the permutation induced by an extension of \(\sigma\) to the \(S\)-matrix field of \(\mathcal{Z}(\mathcal{C})\). Since $\dim(A_{10})=2+(x_{10}+y_{10})d$, $\dim(\hat{\sigma}(A_{10}))=2+(m+n-(x_{10}+y_{10}))d$ by \cite[Lemma 3.6]{schopieray2}, which means $\hat{\sigma}(A_{10})\in\{A_4,A_5,A_{10},A_{11}\}$. If $\hat{\sigma}(A_{10})\in\{A_4,A_5\}$, then $\dim(\hat{\sigma}(A_{10}))=2+(\frac{m+n}{2})d$, so comparing dimensions yields $x_{10}+y_{10}=\frac{m+n}{2}$. If $\hat{\sigma}(A_{10})=A_{10}$, then $\dim(\hat{\sigma}(A_{10}))=\dim(A_{10})$ immediately gives $x_{10}+y_{10}=\frac{m+n}{2}$. If $\hat{\sigma}(A_{10})=A_{11}$, dimension matching gives $(x_{10}+y_{10})+(x_{11}+y_{11})=m+n$. The same conclusions hold for $\hat{\sigma}(A_{11})$ by symmetry.\end{proof}

\begin{prop}
We have the following upper bound:
\begin{equation}
\sum\nolimits_{s} (z_s+z_s')^2 
\leq 16+\frac{15}{8}(m+n)^2.
\label{eqn:SquaresUpperBoundAB}
\end{equation}
\end{prop}
\begin{proof}
By Equations \eqref{eqn:HorizontalSumsAB4}, \eqref{eqn:VerticalSumsAB2}, \eqref{eqn:VerticalSumsAB3}, \eqref{eqn:VerticalSumsAB4} and Lemma \ref{m+n}, we have
\begin{align*}
\sum_{i=6}^9(x_i+y_i)^2&\geq \frac{1}{4}\Bigl(\sum_{i=6}^9(x_i+y_i)\Bigr)^2=\frac{1}{4}(m+n)^2, \quad
\sum_{i=10}^{11}(x_i+y_i)^2\geq \frac{1}{2}\Bigl(\sum_{i=10}^{11}(x_i+y_i)\Bigr)^2= \frac{1}{2}(m+n)^2, \\
\sum_{i=12}^{15}(x_i+y_i)^2 &\geq \frac{1}{4}\Bigl(\sum_{i=12}^{15}(x_i+y_i)\Bigr)^2=\frac{1}{4}(m+n)^2, \quad
\sum_{i=16}^{19}(x_i+y_i)^2 \geq \frac{1}{4}\Bigl(\sum_{i=16}^{19}(x_i+y_i)\Bigr)^2=\frac{1}{4}(m+n)^2.
\end{align*}
Summing the above four inequalities gives
$
\sum_{i=6}^{19}(x_i+y_i)^2  \geq \frac{5}{4}(m+n)^2 
\geq
\frac{5}{4}(m+n)^2 - \frac{1}{8}r^2(m-n)^2
$. Together with  Equation \eqref{mainequal}, we get the desired inequality.
\end{proof}

\begin{rem}\label{all-decompositions}
The upper bound \eqref{eqn:SquaresUpperBoundAB} also holds without the multiplicity-free hypothesis. Keep $A_1,\ldots,A_5$ as before. Label the other distinct constituents of $\cI(h)$ by $A_{12},\ldots,A_{15}$, and the distinct constituents of $\cI(g)\oplus\cI(gh)$ by the remaining indices in $\{6,\ldots,19\}$. There are at most four of the former and at most ten of the latter: the remaining squared multiplicities in $\cI(h)$ sum to $4$, while each of $\cI(g),\cI(gh)$ has either six or three distinct constituents and $[\cI(g),\cI(gh)]=2$. Terms involving unused indices are set to zero. Put $x_i=[\cF(A_i),X]$, $y_i=[\cF(A_i),Y]$, and $Z_s$ as before.
Adjunction and Equation \eqref{FIY} give 
\vspace{-0.2cm}
\begin{align*}
\sum_{i=12}^{15}[\cF(A_i),h]^2&=4,&
\sum_{i=12}^{15}[\cF(A_i),h](x_i+y_i)&=m+n,\\
\sum_{i=6}^{19}\bigl([\cF(A_i),g]+[\cF(A_i),gh]\bigr)^2&=16,&
\sum_{i=6}^{19}\bigl([\cF(A_i),g]+[\cF(A_i),gh]\bigr)(x_i+y_i)&=4(m+n).
\end{align*}
Here the last two sums have zero terms for $12\leq i\leq15$. The Cauchy--Schwarz inequality therefore gives
\[
\sum_{i=12}^{15}(x_i+y_i)^2\geq\frac{(m+n)^2}{4},\qquad
\sum_{i=6}^{11}(x_i+y_i)^2+\sum_{i=16}^{19}(x_i+y_i)^2\geq(m+n)^2.
\]
Equations \eqref{xx}--\eqref{yy}, and hence \eqref{mainequal}, remain valid because $x_i,y_i$ already include all multiplicities. Combining them with these inequalities proves \eqref{eqn:SquaresUpperBoundAB} in every case, without applying Lemma \ref{m+n} to repeated summands.

The twists needed below also hold in every case. The contribution of $A_4,A_5$ to $\Tr(\theta_{\cI(h)})$ is $\dim(\cC)/2$, and the remaining constituents have total dimension $\dim(\cC)/2$. Equality in the triangle inequality forces their twists to be $-1$, and hence $\nu_2(h)=1$. Applying the same argument to the second indicators of $g,gh$ gives $\theta_i^2=\nu_2(g)$ or $\nu_2(gh)$ on their respective constituents. The equality $[\cI(g),\cI(gh)]=2$ shows that these two signs agree.
From this point on, all induction decompositions are allowed, with the indexing convention just described.
\end{rem}

To simplify notations, we denote by 
$\beta := \frac{1}{4}(m+n) - \frac{r}{4}(m-n)$ and 
$\overline{\beta} := \frac{1}{4}(m+n) + \frac{r}{4}(m-n)$.
Observe that  
$\beta + \overline{\beta} =\frac{ m+n}{2}$ 
and 
$\beta^2 + \overline{\beta}^2 = \frac{1}{8}(m+n)^2 + \frac{r^2}{8}(m-n)^2$.

\subsection{The case where $\nu_2(g)=\nu_2(gh)= 1$}

\begin{prop}
Assume $\nu_2(g)=\nu_2(gh)= 1$, then we have the following equality:
\begin{equation}
-\sum\nolimits_s\theta_s^2(z_s{+}z_s')^2=16+\frac{\Delta+8(m+n)}{d}+4(m+n)^2-\sum\nolimits_s(z_s{+}z_s')^2,\label{eqn:UsefulIndicatorAB2}
\end{equation}
where $\Delta:=-(\Tr_{\mathcal{Z}(\cC)}(\theta_{\cI(X)}^2) + \Tr_{\mathcal{Z}(\cC)}(\theta_{\cI(Y)}^2))
\in \{0,\pm2\dim(\cC)=\pm  (16+4(m+n)d)\}$.
\end{prop}

\begin{proof} First assume the decomposition of Proposition \ref{biao}. 
By \cite[Theorem 4.1]{ngschauenburg}, the second Frobenius-Schur indicators  of simple objects $X$ and $Y$ are:
\begin{align*}
\pm \dim(\cC)&=\nu_2(X)\dim(\cC)=\Tr_{\mathcal{Z}(\cC)}(\theta_{\cI(X)}^2)\\
&= x_1\Bigl(1+\frac{m+n}{2}d\Bigr)+x_2(1+\beta d)+x_3(1+\overline{\beta}d)+ x_4\Bigl(2+\frac{m+n}{2}d\Bigr)+x_5\Bigl(2+\frac{m+n}{2}d\Bigr) \\
  &\quad+\sum_{i=6}^9x_i\theta_{i}^2(1+(x_i+y_i)d) +\sum_{i=10}^{11}x_i\theta_{i}^2(2+(x_i+y_i)d) 
   +\sum_{i=12}^{15}x_i(1+(x_i+y_i)d)\\
  &\quad +\sum_{i=16}^{19}x_i\theta_{i}^2(1+(x_i+y_i)d)  + \sum\nolimits_s \theta_s^2 z_s (z_s+z_s')d, 
\end{align*}
\vspace{-0.3cm}
\begin{align*}
\pm \dim(\cC)&=\nu_2(Y)\dim(\cC)=\Tr_{\mathcal{Z}(\cC)}(\theta_{\cI(Y)}^2)\\
&= y_1\Bigl(1+\frac{m+n}{2}d\Bigr)+y_2(1+\beta d)+y_3(1+\overline{\beta}d)+ y_4\Bigl(2+\frac{m+n}{2}d\Bigr)+y_5\Bigl(2+\frac{m+n}{2}d\Bigr) \\
  &\quad+\sum_{i=6}^9y_i\theta_{i}^2(1+(x_i+y_i)d) +\sum_{i=10}^{11}y_i\theta_{i}^2(2+(x_i+y_i)d) 
   +\sum_{i=12}^{15}y_i(1+(x_i+y_i)d)\\
  &\quad +\sum_{i=16}^{19}y_i\theta_{i}^2(1+(x_i+y_i)d)  + \sum\nolimits_s \theta_s^2 z_s' (z_s+z_s')d.
\end{align*}
Adding the above equations together, and notice that $\theta_i=\pm1$ for $6\leq i\leq 11$ and $16\leq i\leq 19$ if $\nu_2(g)=\nu_2(gh)= 1$, it follows from  Equations \eqref{eqn:HorizontalSumsAB4}, \eqref{eqn:HorizontalSumsAB3}, \eqref{eqn:VerticalSumsAB2}, \eqref{eqn:VerticalSumsAB3} and \eqref{eqn:VerticalSumsAB4} that
\begin{equation*}
\Tr_{\mathcal{Z}(\cC)}(\theta_{\cI(X)}^2)+\Tr_{\mathcal{Z}(\cC)}(\theta_{\cI(Y)}^2)
=8(m+n)+\frac{7}{8}(m+n)^2d+\frac{r^2}{8}(m-n)^2d
+\sum_{i=6}^{19}(x_i+y_i)^2d +\sum_s \theta_s^2(z_s{+}z_s')^2d,
\end{equation*}
using Equation \eqref{mainequal} and dividing equation by $d$, we get Equation \eqref{eqn:UsefulIndicatorAB2}.

For a general decomposition, replace the constant term in the expression
for $\dim(A_i)$ by
$\sum_{V\in\{\unit,g,h,gh\}}[\cF(A_i),V]$,
the total multiplicity of the invertible simple objects in $\cF(A_i)$.
By adjunction, we have
\[
\sum_{i=1}^{19}\sum_{V\in\{\unit,g,h,gh\}}
[\cF(A_i),V](x_i+y_i)=8(m+n).
\]
All these constituents have squared twist $1$ by Remark \ref{all-decompositions}. Adding the two trace formulas and using \eqref{mainequal},
we obtain the stated identity for an arbitrary decomposition.
\end{proof}

We use \cite[Proposition 5.6 and Theorem 5.7]{larson} in the following form. For $u,v\in\bbZ$, a sum of roots of unity equal to $u+v\sqrt2$ has at least $|u|+2|v|$ terms. For square-free $t>1$, a sum equal to $u+v\sqrt t$ has at least $|v|\varphi(2t)$ terms, independently of its rational part $u$. Thus \eqref{delta0_2} and \eqref{delta_2} use $2\sum_s(z_s+z_s')^2$ roots, while \eqref{delta0_1} and \eqref{delta_1} use $\sum_s(z_s+z_s')^2$ roots.

\begin{thm}\label{thm1}
If $\nu_2(g)=\nu_2(gh)= 1$ and $m\neq0,n\neq0$, then  $(m,n)$ must be one of $(2,2),(4,6),(6,4)$.
\end{thm}

\begin{proof}
By Equation (\ref{eqn:UsefulIndicatorAB2}) and $\frac{1}{d}=-\frac{1}{4}(m+n)+\frac{1}{4}\sqrt{8 + (m+n)^2}$. If $\Delta = 0$, then we have
\begin{align}
-\sum\nolimits_s\theta_s^2(z_s{+}z_s')^2&=16+\frac{8(m+n)}{d}+4(m+n)^2-\sum\nolimits_s(z_s{+}z_s')^2\nonumber
\\
&=16+2(m+n)^2+2(m+n)\sqrt{8+(m+n)^2}-\sum\nolimits_s(z_s{+}z_s')^2;\label{delta0_1}
\end{align}
if $\Delta \neq 0$, $\frac{\Delta}{d}=\pm4\sqrt{8+(m+n)^2}$, then we get
\begin{align}
-\sum\nolimits_s\theta_s^2(z_s{+}z_s')^2&=16+\frac{\Delta+8(m+n)}{d}+4(m+n)^2-\sum\nolimits_s(z_s{+}z_s')^2\nonumber
\\
&=16+2(m+n)^2+(2(m+n)\pm4)\sqrt{8+(m+n)^2}-\sum\nolimits_s(z_s{+}z_s')^2 \label{delta_1}.
\end{align}
Move the rightmost term in each of the two equations above to the left-hand side, we have
\begin{align}
\Delta=0:\;\;\;\; \sum\nolimits_s(1-\theta_s^2)(z_s{+}z_s')^2&=16+2(m+n)^2+2(m+n)\sqrt{8+(m+n)^2};\label{delta0_2}\\
\Delta \neq 0:\;\;\;\; \sum\nolimits_s(1-\theta_s^2)(z_s{+}z_s')^2&=16+2(m+n)^2+(2(m+n)\pm4)\sqrt{8+(m+n)^2}.\label{delta_2}
\end{align}

When $\Delta = 0.$ Suppose $(m+n)^2+8=2v^2$ for some integer $v>1$.
Then by \cite[Proposition 5.6]{larson}, it requires at least 
\begin{equation*}
16+2(m+n)^2+2(2(m+n))\sqrt{\frac{8 + (m+n)^2}{2}}
\geq 
16+2(m+n)^2+\frac{2}{\sqrt{2}}(2(m+n))(m+n)\\
\end{equation*}
roots of unity to write the right hand side of Equation \eqref{delta0_2}. 
Together with Inequality \eqref{eqn:SquaresUpperBoundAB}, we see
$$
32 + \frac{15}{4}(m+n)^2 \geq  2\sum\nolimits_{s}(z_s + z_s')^2 \geq 16+2(m+n)^2+\frac{2}{\sqrt{2}}(2(m+n))(m+n),$$
which implies $m+n \leq 3$. If  $8 + (m+n)^2=v^2t$, where $v,t$ are integers with $v>1$ and $t>2$ is square free.
Then by \cite[Theorem 5.7]{larson}, it requires at least 
$2(m+n)v\varphi(2t)$ roots of unity to write the right hand side of Equation \eqref{delta0_1}. 
Now by Inequality \eqref{eqn:SquaresUpperBoundAB} and \cite[Lemma 4.2]{schopieray2}, we see
\begin{align*}
16+ \frac{15}{8}(m+n)^2 
&\geq 
\sum\nolimits_s (z_s+z_s')^2 
\geq 
2(m+n)v\varphi(2t)= 
2(m+n) \sqrt{\frac{8 + (m+n)^2}{t}}\varphi(2t)\\&\geq \frac{4}{\sqrt{3}}(m+n)\sqrt{8 + (m+n)^2},
\end{align*}
which implies $m+n \leq 4$.

When $\Delta \neq 0.$ Suppose $(m+n)^2+8=2v^2$ for some integer $v>1$.
Then by \cite[Proposition 5.6]{larson}, it requires at least 
\begin{equation*}
16+2(m+n)^2+2((2(m+n)-4))\sqrt{\frac{8 + (m+n)^2}{2}}
\geq 
16+2(m+n)^2+\frac{2}{\sqrt{2}}((2(m+n)-4))(m+n)\\
\end{equation*}
roots of unity to write the right hand side of Equation \eqref{delta_2}. 
Together with Inequality \eqref{eqn:SquaresUpperBoundAB}, we see
\begin{equation*}
32 + \frac{15}{4}(m+n)^2 \geq  2\sum\nolimits_{s}(z_s + z_s')^2 \geq (16+2(m+n)^2)+\frac{2}{\sqrt{2}}((2(m+n)-4))(m+n),\end{equation*}
which implies $m+n \leq 7$. If   $8 + (m+n)^2=v^2t$, where $v,t$ are integers with $v>1$ and $t>2$ is square free.
Then by \cite[Theorem 5.7]{larson}, it requires at least 
$(2(m+n)-4)v\varphi(2t)$ roots of unity to write the right hand side of Equation \eqref{delta_1}. 
Now by Inequality \eqref{eqn:SquaresUpperBoundAB} and \cite[Lemma 4.2]{schopieray2}, we see
\begin{align}
16+ \frac{15}{8}(m+n)^2 
&\geq 
\sum\nolimits_s (z_s+z_s')^2 
\geq 
(2(m+n)-4)v\varphi(2t)
= 
(2(m+n)-4) \sqrt{\frac{8 + (m+n)^2}{t}}\varphi(2t) \label{3.26}\\&\geq \frac{4}{\sqrt{3}}(m+n-2)\sqrt{8 + (m+n)^2},
\end{align}
which implies $m+n \leq 12$. The value $m+n=11$ is already excluded by Remark \ref{rem:K2abBetaConditions}. For $m+n=12$, one has $v=2$, $t=38$ and $\varphi(76)=36$, so \eqref{3.26} would give $286\geq1440$, a contradiction. Hence $m+n\leq10$.

By Remark \ref{rem:K2abBetaConditions}, when $m+n<10$, $m=n$ and $4\mid(m+n)$, so $(m,n)$ must be one of $(2,2),(4,4),(4,6),(6,4)$.
From the above discussion, if  $(m+n)^2+8=2v^2$ for some integer $v>1$, then $m+n\leq 7$, so  $(m,n)\neq (4,4)$. The proof of the theorem is complete.
 \end{proof}

\subsection{The case where $\nu_2(g)=\nu_2(gh)= -1$}
\begin{prop}
Assume $\nu_2(g)=\nu_2(gh)= -1$, then we  have the following equality:
\begin{equation}
-\sum\nolimits_s (\theta_s+\overline{\theta_s})(z_s{+}z_s')^2
= \frac{3}{4}(m+n)^2+(m+n)\sqrt{8 + (m+n)^2}+\frac{1}{4}r^2(m-n)^2 -2\sum_{i=12}^{15}(x_i+y_i)^2. \label{eqn:UsefulIndicatorAB} 
\end{equation}
\end{prop}
\begin{proof} First assume the decomposition of Proposition \ref{biao}. 
By considering  the first Frobenius-Schur indicators  of  simple objects $X$ and $Y$ \cite{ngschauenburg} , we have 
\begin{align*}
0&=\nu_1(X)=\Tr_{\mathcal{Z}(\cC)}(\theta_{\cI(X)})\\
&= x_1\Bigl(1+\frac{m+n}{2}d\Bigr)+x_2(1+\beta d)+x_3(1+\overline{\beta}d)+ x_4\Bigl(2+\frac{m+n}{2}d\Bigr)+x_5\Bigl(2+\frac{m+n}{2}d\Bigr) \\
&\quad+\sum_{i=6}^9x_i\theta_{i}(1+(x_i+y_i)d) +\sum_{i=10}^{11}x_i\theta_{i}(2+(x_i+y_i)d) 
   -\sum_{i=12}^{15}x_i(1+(x_i+y_i)d)\\
  &\quad +\sum_{i=16}^{19}x_i\theta_{i}(1+(x_i+y_i)d)+ \sum\nolimits_s \theta_s z_s (z_s+z_s')d, 
\end{align*}
\vspace{-0.5cm}
\begin{align*}
0&=\nu_1(Y)=\Tr_{\mathcal{Z}(\cC)}(\theta_{\cI(Y)})\\
&= y_1\Bigl(1+\frac{m+n}{2}d\Bigr)+y_2(1+\beta d)+y_3(1+\overline{\beta}d)+ y_4\Bigl(2+\frac{m+n}{2}d\Bigr)+y_5\Bigl(2+\frac{m+n}{2}d\Bigr) \\
&\quad+\sum_{i=6}^9y_i\theta_{i}(1+(x_i+y_i)d) +\sum_{i=10}^{11}y_i\theta_{i}(2+(x_i+y_i)d) 
   -\sum_{i=12}^{15}y_i(1+(x_i+y_i)d)\\
  &\quad +\sum_{i=16}^{19}y_i\theta_{i}(1+(x_i+y_i)d)+ \sum\nolimits_s \theta_s z'_s (z_s+z_s')d.
\end{align*}
Notice that $\theta_i=\pm i$ for $6\leq i\leq 11$ and $16\leq i\leq 19$ if $\nu_2(g)=\nu_2(gh)= -1$,  together with Equations \eqref{eqn:HorizontalSumsAB4}, \eqref{eqn:HorizontalSumsAB3} and \eqref{eqn:VerticalSumsAB3},
we obtain
\begin{align*}
0=&\Tr_{\cZ(\cC)}(\theta_{\cI(X)})+\Tr_{\cZ(\cC)}(\theta_{\cI(Y)})+\overline{\Tr_{\cZ(\cC)}(\theta_{\cI(X)})+\Tr_{\cZ(\cC)}(\theta_{\cI(Y)})}\\
=&4(m+n)+\frac{7}{4}(m+n)^2d+\frac{1}{4}r^2(m-n)^2d-2\sum_{i=12}^{15}(x_i+y_i)^2d+\sum\nolimits_s (\theta_s+\overline{\theta_s})(z_s{+}z_s')^2d.
\end{align*}
Dividing both sides of the above equation by 
$d$ and rearranging, we obtain Equation \eqref{eqn:UsefulIndicatorAB}.

For a general decomposition, the constituents of $\cI(g)\oplus\cI(gh)$ have twists $\pm i$ by Remark \ref{all-decompositions}, so they do not contribute to the real part. The integral parts of the contributions from $A_1,\ldots,A_5$ sum to $3(m+n)$, whereas
\[
\sum_{i=12}^{15}[\cF(A_i),h](x_i+y_i)=m+n.
\]
The twists of the latter objects are $-1$. Their coefficients of $d$ are still $(x_i+y_i)$, so the same identity holds with arbitrary multiplicities.

\end{proof}

\begin{lem}\label{ineq}
Let \(t\) be a square-free positive integer.  If a prime $p$ divides $t$ and $p\geq 19$, then
$
\frac{\varphi(2t)}{\sqrt{t}}>4.
$
\end{lem}
\begin{proof}
Since \(t\) is square-free, we may write
$t=2^{\varepsilon}\prod_{i=1}^r q_i$,
where \(\varepsilon\in\{0,1\}\) and \(q_1,\dots,q_r\) are distinct odd primes.
Then $\varphi(2t)=\varphi(2^{\varepsilon+1})\prod_{i=1}^r \varphi(q_i)=2^\varepsilon\prod_{i=1}^r(q_i-1)$.
As $\varepsilon\in\{0,1\}$, it follows that
\begin{equation*}
\frac{\varphi(2t)}{\sqrt{t}}=2^{\varepsilon/2}\prod_{i=1}^r\frac{q_i-1}{\sqrt{q_i}}\geq
\prod_{i=1}^r\frac{q_i-1}{\sqrt{q_i}}.\end{equation*}
And for every odd prime $q_i$,
we have $\frac{q_i-1}{\sqrt{q_i}}\ge 1$,
hence $\prod_{i=1}^r\frac{q_i-1}{\sqrt{q_i}}\geq\frac{p-1}{\sqrt{p}}$. Therefore, $\frac{\varphi(2t)}{\sqrt{t}}\geq
\frac{p-1}{\sqrt{p}}>4$.
\end{proof}

\begin{lem}\label{lem:pell-three-eleven}
Let $\cC$ be a spherical fusion category with Grothendieck ring
$R(m,n)$, where $m,n>0$. Then $(m, n)$ cannot arise from the equation
$
8+(m+n)^2=11v^2.
$
\end{lem}

\begin{proof}
Suppose that $q^2 + 8 = 11v^2$ for some $v \in \mathbb{Z}_{>0}$, where $q = m + n$. Then $\mathbb{Q}(\dim(X) \mid X \in \mathcal{O}(\mathcal{C})) = \mathbb{Q}(\dim(\mathcal{C})) = \mathbb{Q}(\sqrt{11})$. The nontrivial quadratic character of this field has conductor $44$. Put $q=m+n$ and retain the notation
$d,r,f_i,A_i,\cF,\cI$ introduced above. In particular,
$\dim(\cC)=f_1=q^2+8+q\sqrt{q^2+8},
$
and its Galois conjugate is
$f_2=q^2+8-q\sqrt{q^2+8}.
$
Since $f_1f_2=8(q^2+8)$, we have $
\frac1{f_1}-\frac1{f_2}
=-\frac{q}{4\sqrt{q^2+8}}$.
Therefore, \cite[Theorem 5.1]{schopieray-dimension-block} gives
$
\frac{q}{4v\sqrt{11}}
=
\left|\frac{1}{f_1}-\frac{1}{f_2}\right|
\leq \frac{1}{2\sqrt{11}},
$
which implies $q\leq 2v$. Consequently,
$
8=11v^2-q^2\geq 7v^2.
$
Since $v$ is a positive integer, we obtain $v=1$,
and hence $q^2=3$, a contradiction.
\end{proof}

\begin{thm}\label{thm2}
If $\nu_2(g)=\nu_2(gh)= -1$ and $m\neq0,n\neq0$,   then $(m,n)$ satisfies one of the following three Pell equations:
\begin{equation*}
8+(m+n)^2=2v^2,  \quad  8+(m+n)^2=6v^2, \quad
8+(m+n)^2=3v^2,  \quad   
\end{equation*}
where $v$ is a positive integer greater than $1$. In particular, if $m = n$, then $(m, m)$ can only arise from the first two equations.
\end{thm}

\begin{proof}
Assume $8 + (m+n)^2=v^2t$, where $v,t$ are integers with $v>1$ and $t$ is square free. 
By Remark \ref{rem:K2abBetaConditions}, $\frac{3}{4}(m+n)^2\in \bbZ$ and $\frac{r^2}{4}(m-n)^2\in \bbZ$, then it follows from \cite[Theorem 5.7]{larson} that it requires at least 
$(m+n)v\varphi(2t)$ roots of unity to express the right hand side of Equation \eqref{eqn:UsefulIndicatorAB}. Let $q:=m+n$, then we have 
\begin{equation*}
32+ \frac{15}{4}q^2 
\geq 
2\sum\nolimits_s (z_s+z_s')^2 
\geq 
qv\varphi(2t)
= 
q \sqrt{\frac{8 + q^2}{t}}\varphi(2t),
\end{equation*}
by Inequality \eqref{eqn:SquaresUpperBoundAB}. Thus, we obtain
\begin{equation}
  \frac{\varphi(2t)}{\sqrt t}
\le
\frac{\frac{32}{q^2}+\frac{15}{4}}
{\sqrt{1+\frac8{q^2}}}.
\label{t2}  
\end{equation}
By Remark \ref{rem:K2abBetaConditions},  either $q=4,8$, or $q\ge10$. If $q=4$, then
$t=6$; if $q=8$, then $t=2$.
Assume $q\ge10$ below. Since the right-hand side of Equation \eqref{t2} is decreasing for \(q>10\),
\begin{equation*}
\frac{\varphi(2t)}{\sqrt t}
\le
\frac{\frac{32}{100}+\frac{15}{4}}
{\sqrt{1+\frac8{100}}}
=
\frac{407}{100\sqrt{27/25}}
<3.92.
\end{equation*}
By Lemma \ref{ineq}, we know that prime divisors of $t$ belong  to $\{2,3,5,7,11,13,17\}$. 
A direct inspection shows that the set of  square-free integers satisfying $\frac{\varphi(2t)}{\sqrt t}<3.92$
is
\begin{equation*}
 t\in\{1,2,3,5,6,7,10,11,13,14,15,17,21,30,33,39,42\}.  
\end{equation*}
We know $t>1$ by Lemma \ref{lemfirst}.

Assume $t$ is odd.  Lemma \ref{lem:pell-three-eleven} implies $t\neq 11$.  Notice that we have 
$q^2-tv^2=-8$.
 Write $q=2u$ as $q$ is even,  then $v$ is also even, let $v=2w$, and 
$u^2-tw^2=-2$. 
For an odd prime factor $p$ of $t$, we have 
$u^2\equiv-2\pmod p.$
Hence $\left(\frac{-2}{p}\right)=1$, we have $p\equiv1,3\pmod8$ by the quadratic reciprocity law, thus 
$t\in\{3,17,33\}$.
If $t=17,33$, then $t\equiv1\pmod 8$, and 
$u^2-w^2\equiv6\pmod8$, which has no solution  as the quadratic residues modulo \(8\) are \(0,1,4\). Hence  $t=3$.
Now suppose \(t\) is even. 
If \(5\mid t\), then  $q^2\equiv-8\equiv2\pmod5$, it is
impossible as $2$ is not a quadratic residue modulo $5$. 
 A similar argument shows $7\nmid t$, therefore
$t\in\{2,6\}$. In summary, $t\in\{2,3,6\}$  

Next, we show that 
 $4\mid q$ for $t=2,6$ and that $4\nmid q$ for $t=3$. If so,  Remark \ref{rem:K2abBetaConditions} implies   that if $m = n$, then $m + n$ can only come from the Pell equations of the form $8+q^2=2v^2$ and $8+q^2=6v^2$. Here we give a proof for $t=2,3$, the proof for the other is the same.
For equation $8+q^2=2v^2$, $q$ must be  even, let $q = 2k$. Then   $4 + 2k^2 = v^2$, so $v= 2w$ is also even, which means  $4\mid q$. For  equation $8+q^2=3v^2$, we have  $q^2 \equiv 3v^2 \equiv -v^2 \pmod{4}$. If $v$ is odd, then $q^2\equiv-v^2 \equiv 3 \pmod{4}$, it  is impossible. Thus $v$ is even  and $q^2 \equiv 0 \pmod 4$, so $q$ is even. If $q = 4k$, then  $3v^2 \equiv 8 \pmod{16}$, hence $v^2 \equiv 8 \pmod{16}$, impossible.
\end{proof}

\begin{rem}\label{pell}
We note that each of the three Pell equations in Theorem \ref{thm2} has infinitely many solutions.
The set of all solutions with 
$m+n\leq1000$ is $\{2,4,8,10,38,44,48,142,280,436,530\}$. Moreover, if $m+n\leq 100$, the pairs $(m,n)$ not excluded by our arguments are $(0,0),\allowbreak(1,0),\allowbreak(0,1),\allowbreak(2,0),\allowbreak(0,2),\allowbreak(2,2),\allowbreak(4,4),\allowbreak(4,6),\allowbreak(6,4),\allowbreak(18,20),\allowbreak(20,18),\allowbreak(22,22),\allowbreak(20,24),\allowbreak(24,20),\allowbreak(24,24)$.
\end{rem}

\begin{thm}
\label{thm:PseudounitaryBoundselfdual}
Let $\cC$ be a  fusion category with the fusion rules \ref{selfdual}.
Then $(m,n)$ must be equal to one of
$(0,0),(1,0),(0,1),(2,0),(0,2),(2,2),(4,4),(4,6),(6,4)$ if $m+n<38$. 
\end{thm}
\begin{proof}
If either $m$ or $n$ is zero, then  $(m,n)$ must be one of $(0,0),(0,1),(1,0),(0,2),(2,0)$ by Proposition \ref{notzero}.
Otherwise, if $\nu_2(g)=\nu_2(gh)= 1$, then $(m,n)$ must be one of $(2,2),(4,6),(6,4)$ by Theorem \ref{thm1};
 and if $\nu_2(g)=\nu_2(gh)= -1$ and $m+n<38$,   then  $(m,n)$ must be one of  $(2,2),(4,4),(4,6),(6,4)$ by Theorem \ref{thm2}, Remark \ref{rem:K2abBetaConditions} and Remark \ref{pell}.
\end{proof}


\section{The Non-Self-Dual Case}\label{section-nonself-dual}

Given a non-self-dual $\bbZ/2\bbZ\times\bbZ/2\bbZ$-quadratic fusion category $\cC$ of rank six, let $\cO(\cC)=\{\unit_{\cC},g,h,gh,X,X^*\}$ and fusion rules are determined by
\begin{align}
g^2= h^2 = \unit_{\cC}, \quad g \otimes X = X^*= X \otimes g,\quad
X \otimes X^* = \unit_{\cC} \oplus h \oplus m X \oplus m X^*= X^* \otimes X.  \label{notselfdual} \tag{$R'(m)$}
\end{align}
Let us write \ref{notselfdual} for such a fusion ring.  
Set $d:= \dim(X)$. Then $d^2 = 2+2md$, $d = m + \sqrt{2 + m^2}$
and the formal codegrees of $\cC$ are:
\begin{equation}
f_1 = 8 + 4m^2 + 2m\sqrt{8 + 4m^2},\;
f_2 = 8 + 4m^2 - 2m\sqrt{8 + 4m^2},\;
f_3 =f_4= 8,\;
f_5 = f_6=4.
\end{equation}

Same as the self-dual case, we have the following lemma. 
\begin{lem}
 There are distinct simple objects 
$ \unit_{\mathcal{Z}(\cC)},A_1,A_2,A_3,A_4,A_5 \in \cO(\cZ(\cC))$
such that
$
\cI(\unit_\cC) = \unit_{\mathcal{Z}(\cC)} \oplus A_1\oplus A_2\oplus A_3\oplus A_4\oplus A_5$ with \begin{equation*}
\dim(A_1) = 1 + md,\;\;
\dim(A_2) = 
\dim(A_3) =1+\frac{m}{2}d,\;\;
\dim(A_4)=\dim(A_5) = 2 + md.
\end{equation*}
\end{lem}

\begin{rem} \label{remark-m}
\label{K2abBetaConditions}
Since $\dim(A_2)$ and $\dim(A_3) \in \bbZ[d]$,  $2\mid m$. 
\end{rem}
If $m=0$, $R'(0)$ can be categorified by Lemma \ref{R(0)}. In what follows, we always assume $m\neq 0$. The same induction calculations, including Remark \ref{all-decompositions}, apply with $m+n$ replaced by $2m$ and $m-n=0$. We use the same notations as Subsection \ref{subsection-decomposition}.
 
\subsection{The case where $\nu_2(g)=\nu_2(gh)= 1$}

\begin{prop}
If $\nu_2(g)=\nu_2(gh)= 1$, then we have the following equality:
\begin{align}
-\sum\nolimits_s\theta_s^2(z_s{+}z_s')^2=16+\frac{16m}{d}+16m^2-\sum\nolimits_s(z_s{+}z_s')^2. \label{11}
\end{align}
\end{prop}

\begin{proof} First assume the decomposition of Proposition \ref{biao}. 
It follows from  \cite[Theorem 4.1]{ngschauenburg} that the second Frobenius-Schur indicators  of  $X$ and $X^*$ are:
\begin{align*}
0=\nu_2(X)\dim(\cC)&=\Tr_{\mathcal{Z}(\cC)}(\theta_{\cI(X)}^2)\\
&=x_1(1+md)+x_2\Bigl(1+\frac{m}{2}d\Bigr)+x_3\Bigl(1+\frac{m}{2}d\Bigr)+x_4(2+md)+x_5(2+md)\\
&\quad+\sum_{i=6}^9x_i\theta_{i}^2(1+(x_i+y_i)d) +\sum_{i=10}^{11}x_i\theta_{i}^2(2+(x_i+y_i)d) 
   +\sum_{i=12}^{15}x_i(1+(x_i+y_i)d)\\
  &\quad +\sum_{i=16}^{19}x_i\theta_{i}^2(1+(x_i+y_i)d)  + \sum\nolimits_s \theta_s^2 z_s (z_s+z_s')d, 
\end{align*}
\vspace{-0.5cm}
\begin{align*}
0=\nu_2(X^*)\dim(\cC)&=\Tr_{\mathcal{Z}(\cC)}(\theta_{\cI(X^*)}^2)\\
&= y_1(1+md)+y_2\Bigl(1+\frac{m}{2}d\Bigr)+y_3\Bigl(1+\frac{m}{2}d\Bigr)+y_4(2+md)+y_5(2+md)\\
  &\quad+\sum_{i=6}^9y_i\theta_{i}^2(1+(x_i+y_i)d) +\sum_{i=10}^{11}y_i\theta_{i}^2(2+(x_i+y_i)d) 
   +\sum_{i=12}^{15}y_i(1+(x_i+y_i)d)\\
  &\quad +\sum_{i=16}^{19}y_i\theta_{i}^2(1+(x_i+y_i)d)  + \sum\nolimits_s \theta_s^2 z'_s (z_s+z_s')d.
\end{align*}
Adding the above equations together, using  Equations \eqref{eqn:HorizontalSumsAB4}, \eqref{eqn:HorizontalSumsAB3}, \eqref{eqn:VerticalSumsAB2}, \eqref{eqn:VerticalSumsAB3} and \eqref{eqn:VerticalSumsAB4}, and noticing that $\theta_i=\pm1$ for $6\leq i\leq 11$ and $16\leq i\leq 19$ when $\nu_2(g)=\nu_2(gh)= 1$. We have
\begin{equation*}
0= 16m+\frac{7}{2}m^2d
+\sum_{i=6}^{19}(x_i+y_i)^2d +\sum\nolimits_s \theta_s^2(z_s{+}z_s')^2d,
\end{equation*}
using Equation \eqref{mainequal} and dividing equation by $d$, we get the Equation \eqref{11}.
The extension to arbitrary induction multiplicities is the same as in Remark \ref{all-decompositions} and the corresponding indicator calculation in Section \ref{section-selfdual}, with $m+n=2m$ and $m-n=0$.
\end{proof}

\begin{thm}\label{thm3}
If $\nu_2(g)=\nu_2(gh)= 1$, then $R'(m)$ can not be categorified for all $m>0$.
\end{thm}

\begin{proof}
By Equation \eqref{11} and $\frac{1}{d} = \frac{\sqrt{2+m^{2}} - m}{2}
$, we get
\begin{equation}
-\sum\nolimits_s\theta_s^2(z_s{+}z_s')^2=16+\frac{16m}{d}+16m^2-\sum\nolimits_s(z_s{+}z_s')^2=16+8m^2+8m\sqrt{2+m^2}-\sum\nolimits_s(z_s{+}z_s')^2,\label{delta0_11}
\end{equation}
which then implies
\begin{equation}
\sum\nolimits_s(1-\theta_s^2)(z_s{+}z_s')^2=16+8m^2+8m\sqrt{2+m^2}.\label{delta0_22}
\end{equation}
Notice that $2\sum_s(z_s{+}z_s')^2\geq | \sum_s(1-\theta_s^2)(z_s{+}z_s')^2|$. Using Equation \eqref{delta0_22} and Inequality \eqref{eqn:SquaresUpperBoundAB}, we get
\begin{equation*}
32 + 15m^2 \geq 16+8m^2+8m\sqrt{2+m^2},
\end{equation*}
which implies $m \leq 2$, so  $m=2$ by Remark \ref{remark-m}.

Assume $2 + m^2=v^2t$ where $v,t$ are integers with $v>0$ and $t$ is square free. 
Then by \cite[Theorem 5.7]{larson}, it requires at least 
$8mv\varphi(2t)$ roots of unity to write the right hand side of Equation \eqref{delta0_11}. 
Now by Inequality \eqref{eqn:SquaresUpperBoundAB}, we see
\begin{equation*}
16+ \frac{15}{2}m^2 
\geq 
\sum\nolimits_s (z_s+z_s')^2 
\geq 
8mv\varphi(2t)
= 
8m \sqrt{\frac{2 + m^2}{t}}\varphi(2t).
\end{equation*}
But $m=2$ does not satisfy this inequality. This completes the proof of the theorem.
 \end{proof}

\subsection{The case where $\nu_2(g)=\nu_2(gh)= -1$}

\begin{prop}\label{-1double}
If $\nu_2(g)=\nu_2(gh)= -1$, then we  have the following equality:
\begin{equation}
 -\sum\nolimits_s (\theta_s+\overline{\theta_s})(z_s+z'_s)^2 
  =  3m^2 + 2m\sqrt{8+4m^2}-2\sum_{i=12}^{15}(x_i+y_i)^2.\label{eqn:UsefulIndicatorAB1}
  \end{equation}\end{prop}
\begin{proof} First assume the decomposition of Proposition \ref{biao}. 
We first consider  the first Frobenius-Schur indicators  of  simple objects $X$ and $X^*$ \cite{ngschauenburg}, we have
\begin{align*}
0=\nu_1(X)=&\Tr_{\mathcal{Z}(\cC)}(\theta_{\cI(X)})
= x_1(1+md)+x_2\Bigl(1+\frac{m}{2}d\Bigr)+x_3\Bigl(1+\frac{m}{2}d\Bigr)+ x_4(2+md)+x_5(2+md) \\
  &\quad+\sum_{i=6}^9x_i\theta_{i}(1+(x_i+y_i)d) +\sum_{i=10}^{11}x_i\theta_{i}(2+(x_i+y_i)d) 
   -\sum_{i=12}^{15}x_i(1+(x_i+y_i)d)\\
  &\quad +\sum_{i=16}^{19}x_i\theta_{i}(1+(x_i+y_i)d)  + \sum\nolimits_s \theta_s z_s (z_s+z_s')d, 
\end{align*}
\vspace{-0.5cm}
 \begin{align*}
0=\nu_1(X^*)=&\Tr_{\mathcal{Z}(\cC)}(\theta_{\cI(X^*)})
= y_1(1+md)+y_2\Bigl(1+\frac{m}{2}d\Bigr)+y_3\Bigl(1+\frac{m}{2}d\Bigr)+ y_4(2+md)+y_5(2+md) \\
  &\quad+\sum_{i=6}^9y_i\theta_{i}(1+(x_i+y_i)d) +\sum_{i=10}^{11}y_i\theta_{i}(2+(x_i+y_i)d) 
   -\sum_{i=12}^{15}y_i(1+(x_i+y_i)d)\\
  &\quad +\sum_{i=16}^{19}y_i\theta_{i}(1+(x_i+y_i)d)  + \sum\nolimits_s \theta_s z'_s (z_s+z_s')d.
\end{align*}
Notice that $\theta_i=\pm i$ for $6\leq i\leq 11$ and $16\leq i\leq 19$ if $\nu_2(g)=\nu_2(gh)= -1$, together with Equations \eqref{eqn:HorizontalSumsAB4}, \eqref{eqn:HorizontalSumsAB3} and \eqref{eqn:VerticalSumsAB3},
we obtain
\begin{align*}
0=&\Tr_{\cZ(\cC)}(\theta_{\cI(X)})+\Tr_{\cZ(\cC)}(\theta_{\cI(X^*)})+\overline{\Tr_{\cZ(\cC)}(\theta_{\cI(X)})+\Tr_{\cZ(\cC)}(\theta_{\cI(X^*)})}\\
= &8m+7m^2d-2\sum_{i=12}^{15}(x_i+y_i)^2d+\sum\nolimits_s (\theta_s+\overline{\theta_s})(z_s{+}z_s')^2d.
\end{align*}
Dividing both sides of the above equation by 
$d$ and rearranging, we obtain Equation \eqref{eqn:UsefulIndicatorAB1}.
The extension to arbitrary induction multiplicities is the same as in Remark \ref{all-decompositions} and the corresponding indicator calculation in Section \ref{section-selfdual}, with $m+n=2m$ and $m-n=0$.
\end{proof}

\begin{thm}\label{thm4}
If  $\nu_2(g)=\nu_2(gh)= -1$ and 
$m\neq 0$, then $m$ satisfies one of  the following two Pell equations
\begin{equation*}
8+4m^2=2v^2,  \quad  8+4m^2=6v^2,
\end{equation*}
where $v$ is a positive integer greater than $1$.
\end{thm}

\begin{proof}
Let $8 + 4m^2=v^2t$, where $t,v$ are integers with $v>1$ and $t$ is square free.
Then it requires at least 
$2mv\varphi(2t)$ roots of unity to write the right hand side of Equation \eqref{eqn:UsefulIndicatorAB1} by \cite[Theorem 5.7]{larson}. 
Now by Inequality \eqref{eqn:SquaresUpperBoundAB}, we have
\begin{equation*}
32+ 15m^2
\geq 
2\sum\nolimits_s (z_s+z_s')^2 
\geq 
2mv\varphi(2t)
= 
2m \sqrt{\frac{8 + 4m^2}{t}}\varphi(2t).
\end{equation*}
Let \(q:=2m\), then
$
q^2+8=v^2t
$.
If \(q=4\), then \(t=6\); if \(q=8\), then \(t=2\). Hence, in both cases,
$
t\in\{2,6\}$, it suffices to consider the case \(q\ge12\). As $\frac{32}{q^2}\le\frac{2}{9}$ and $\sqrt{1+\frac8{q^2}}>1$,
  we have 
\begin{equation*}
\frac{\varphi(2t)}{\sqrt t}
\leq \frac{32+15m^2}{2m\sqrt{8+4m^2}}<\frac{32+\frac{15q^2}{4}}{q^2}\leq
\frac{143}{36}
<4.
\end{equation*}
By using the same argument as Theorem \ref{thm2},  we know that prime divisors of $t$ belong  to $\{2,3,5,7,11,13,17\}$ and  that the set of square-free integers satisfying $\frac{\varphi(2t)}{\sqrt t}
<4 $ is 
\begin{equation}
\{1,2,3,5,6,7,10,11,13,14,15,17,21,30,33,39,42\}.
\label{s2}    
\end{equation}
Since  \(4\mid q\), write $q=4u$, and 
\begin{equation}
16u^2-tv^2=-8.
\label{s4}    
\end{equation}
If $t$ is odd, then reducing Equation \eqref{s4} modulo \(8\), we obtain
$-tv^2\equiv0\pmod8$, that is, $v^2\equiv0\pmod8$ as $t$ is odd,  so $v$ is even. Let  $v=2w$, then
$4u^2-tw^2=-2$, reducing it modulo \(4\), we get
$-tw^2\equiv2\pmod4$. However,   quadratic residues modulo \(4\) are \(0,1\), the left-hand side can only be congruent to \(0,1,\) or \(3\pmod4\), it is a contradiction. Hence \(t\) is even.
Thus, $t\in\{2,6,10,14,30,42\}$.

If \(5\mid t\), then  
$q^2\equiv-8\equiv2\pmod5$, 
it is impossible as \(2\) is not a quadratic residue modulo \(5\); a similar argument shows $7\nmid t$.  Therefore
$t\in\{2,6\}$.
\end{proof}

\begin{rem}\label{pell2}
Each of the two Pell equations in Theorem \ref{thm4} has infinitely many solutions.
The set of all solutions with 
$m\leq 1000$ is $\{2,4,22,24,140,218,816\}.$ 
\end{rem}

\begin{thm}\label{thm34}
\label{thm:PseudounitaryBoundnonselfdual}
Let $\cC$ be a  fusion category with the fusion rules \ref{notselfdual}.
Then 
$m\in\{0,2,4\}$ if $m<22$.
\end{thm}
\begin{proof}
If $m=0$, $R'(0)$ can be categorified by Lemma \ref{R(0)}. Assume $m>0$ below. If $\nu_2(g)=\nu_2(gh)= 1$, then no such categorification exists
 by Theorem \ref{thm3}; if $\nu_2(g)=\nu_2(gh)= -1$ and $m<22$, then $m$ must be one of $2,4$ by Theorem \ref{thm4} and Remark \ref{pell2}.
\end{proof}

\section{Categorifications of fusion rings}\label{section-categorification}
In this section, we give categorification of fusion rings $R(m,n)$ and $R'(m)$. We begin with the self-dual case.

Let $\mathcal{C}$ be a modular fusion category. Assume that $\mathcal{C}$ contains  a Tannakian fusion subcategory $\Rep(G)$, then there is another modular fusion category, denoted by $\cC_G^0$, obtained from the de-equivariantization of centralizer of $\Rep(G)$ in $\mathcal{C}$ by $\Rep(G)$. The dimension of  $\mathcal{C}_G^0$ is given by $\dim(\mathcal{C}_G^0)=\frac{\dim(\mathcal{C})}{|G|^2}$ \cite{drinfeld2010braided}. 
In addition, let $A=\Fun(G)\in \Rep(G)$ be the regular algebra of $\Rep(G)$, then it is a connected \'{e}tale algebra in sense of \cite{dmno,kirillovostrik}, and   the category of local  modules $\mathcal{C}_A^0$ over  the \'{e}tale algebra $A$  in $\mathcal{C}$ is exactly the modular fusion category $\mathcal{C}_G^0$.

Let $k$ be a positive integer. Let $\cC(\mathfrak{sl}_4,k)$ be the modular fusion category obtained from semisimplification of the category $\Rep(U_q(\mathfrak{sl}_4))$ of  finite-dimensional tilting modules of $U_q(\mathfrak{sl}_4)$ at the standard level-$k$ root of unity (with the corresponding normalization of $q$), we refer the readers to \cite{bakalovkirillov,egno2015} for a detailed construction of modular fusion categories $\cC(\mathfrak{g},k)$, where $\mathfrak{g}$ is a semisimple Lie algebra over $\bbK$. The  modular fusion category $\cC(\mathfrak{sl}_4,8)$ contains a Tannakian subcategory $\Rep(\bbZ/4\bbZ)$, and there is a braided tensor equivalence  
\begin{align*}\cZ(\cC_{4,8,4}) \cong \cC(\mathbb Z/2\mathbb Z\times \mathbb Z/2\mathbb Z ,\eta)\boxtimes \cC(\mathfrak{sl}_4,8)^0_{\mathbb Z/4\mathbb Z},
\end{align*}
where $\cC_{4,8,4}$ categorifies the fusion ring $R(4,6)$ or $R(6,4)$ by \cite[Theorem 5.2]{edie-michell}   and $\cC(\mathbb Z/2\mathbb Z\times \mathbb Z/2\mathbb Z ,\eta)$ is a modular   fusion category determined   by the metric group $(\mathbb Z/2\mathbb Z\times\mathbb Z/2\mathbb Z,\eta)$ \cite[Appendix A]{drinfeld2010braided}.  

Combining the above arguments with Proposition \ref{notzero}, we have the following:
\begin{prop}\label{selfdual-categorification}
    For any $(m,n)\in\{(0,0),(0,1),(1,0),(0,2),(2,0),(4,6),(6,4)\}$, there exists a fusion category $\cC$ such that $\cK_0(\cC)=R(m,n)$.
\end{prop}
Next we consider the categorifications of fusion rings $R'(m)$. 
\begin{lem}\label{R(0)}
There exists a fusion category $\cB$ such that $\cK_0(\cB)=R'(0)$. 
\end{lem}
\begin{proof}
Let $\mathcal{D}$  be the Ising category  with simple objects $\cO(\cD)=\{\unit_\cD,\chi,V\}$, fusion rules of $\cD$ are given by $\chi^2=\unit_\cD$, $\chi \otimes V=V$ and $V\otimes  V=\unit_\cD\oplus\chi$, and let $a$ generate $\cC=\mathrm{Vec}_{\bbZ/4\bbZ}$. Then it is easy to see that $\cD\boxtimes\cC$ contains a rank six fusion subcategory $\cB$ with  $
\cO(\cB)=\{\unit_\cD\boxtimes\unit_\cC,\chi\boxtimes\unit_\cC,\unit_\cD\boxtimes a^2, \chi\boxtimes a^2, V\boxtimes a,V\boxtimes a^3\}$. Set $X=V\boxtimes a$, $h=\chi\boxtimes\unit_\cC$ and $g=\unit_\cD\boxtimes a^2$. Then
\[
X^*=V\boxtimes a^3,\quad X\otimes X^*=\unit\oplus h,\quad X\otimes X=g\oplus gh,\quad g\otimes X=X^*.
\]
Thus $\cK_0(\cB)=R'(0)$. 
\end{proof}
In the rest of this section,  we show that the fusion ring $R'(2)$ can be categorified by fusion category related to  near-group category. Let $\cC$  be the near-group fusion category of type $\bbZ/2\bbZ \times \bbZ/4\bbZ+8$ discussed in \cite[section 3.2]{rowellsolomonzhang}. It follows from \cite[Theorem 3.2]{rowellsolomonzhang} that $\cZ(\cC)$ contains a Tannakian fusion category $\Rep(\bbZ/2\bbZ)$ and  $\cZ(\cC)_{\mathbb Z/2\mathbb Z}^0\cong \cD  \boxtimes \cC(\mathbb Z/2\mathbb Z,q)$ as braided fusion categories, where $\cD$ is a spin modular category of rank $18$, i.e., $\cD$ contains a super-modular fusion category $\cD_0$ of rank $10$,  $\cC(\bbZ/2\bbZ,q)$ is a pointed modular fusion category of dimension $2$ and $q$ is a non-degenerate quadratic form on $\mathbb Z/2\mathbb Z$. 
Note $\dim(\cZ(\cC)_{\mathbb Z/2\mathbb Z}^0)=\frac{\dim(\cZ(\cC))}{4}=960+384\sqrt{6}$ and $\dim(\cD)=480+192\sqrt{6}=48(\chi_6^2)^2$, where $\chi_m^n:=n+\sqrt{m}$ following the notation in \cite{rowellsolomonzhang}.
\begin{lem}
 Let $\cC$ be a near-group fusion category of type $\bbZ/2\bbZ\times\bbZ/4\bbZ+8$. Then $\cC\cong\cA^{\bbZ/2\bbZ}$, where $\cA$ is a fusion category of rank $6$ with four invertible simple objects; moreover, $\cK_0(\cA)$ is commutative.
\end{lem}
\begin{proof}
Since $\cC$ is trivially graded, the Tannakian fusion category $\Rep(\bbZ/2\bbZ)$ of $\cZ(\cC)$ must be embedded into $\cC$ via the forgetful functor $\cF$, hence $\cC\cong\cA^{\bbZ/2\bbZ}$ by \cite[Proposition 2.10]{etingof2011weakly}. The tensor action of the corresponding order-two invertible object has four size-two orbits on the invertible simples of $\cC$, while its unique non-invertible simple object is fixed and splits into two distinct simples of half its dimension under de-equivariantization. By \cite[Proposition 4.26]{drinfeld2010braided}, the dimensions of simple objects of $\cA$ are therefore $1,\chi_6^2$, notice that the expression  $\dim(\cA)=24+8\sqrt{6}=4+2(\chi_6^2)^2$  is unique, so $\text{rank}(\cA)=6$. If $\cK_0(\cA)$ is not commutative, then there is a unique irreducible representation $\phi$ of dimension $2$, and  $\frac{1}{\dim(\cA)}+\frac{1}{\sigma(\dim(\cA))}+\frac{2}{f_\phi}=1$, where $f_\phi$ is the formal codegree corresponding to $\phi$, The first two terms sum to $1/4$, so $f_\phi=8/3$, which is impossible. Thus, $\cK_0(\cA)$ must be commutative.
\end{proof}
Let $b=F(X_{(0,1)})$ be the non-trivial simple object of the pointed factor $\cC(\bbZ/2\bbZ,q)$, so $\theta_b=i$. 
Following \cite[Theorem 3.2]{rowellsolomonzhang}, with this correction to the representative in Table 7, the isomorphism classes of simple objects of $\cD_0$ and $\cD$ are denoted by 
 \begin{align*}
 \cO(\cD_0)=\{&F(\unit),F(Y_{(0,0)}),(Z_{(1,1),(1,3)})_1,(Z_{(1,1),(1,3)})_2,F(W_3),\\&F(X_{(1,0)}),F(Y_{(1,0)}),(Z_{(0,1),(0,3)})_1,(Z_{(0,1),(0,3)})_2,F(W_1)\},\\
\cO(\cD)=\cO(\cD_0)\cup\{
 &F(Z_{(0,0),(1,0)}),F(Z_{(0,1),(1,3)}),F(W_{21}),F(W_{23}),(W_{25})_1,(W_{25})_2,(W_{26})_1,(W_{26})_2\},
\end{align*}
where $F$ is the tensor functor determined by regular algebra of $\Rep(\bbZ/2\bbZ)$. Moreover, the dimensions and twists are as in \cite[Table 7]{rowellsolomonzhang}, with the correction just explained; we include them here for later use.
\begin{center}
\begin{tabular}{|c|c|c|}
\hline
\textbf{dim} & \textbf{Objects} & \textbf{Twists}  \\ \hline
1& $F(\unit)$, $F(X_{(1,0)})$&1,-1\\
\hline
$\chi_{24}^5$& $F(Y_{(0,0)})$,$ F(Y_{(1,0)})$&1,-1\\
\hline
$2\chi_6^3$& $F(Z_{(0,0),(1,0)}), F(Z_{(0,1),(1,3)})$&$1,-i$\\
\hline
$\chi_6^3$&$(Z_{(1,1),(1,3)})_1,(Z_{(1,1),(1,3)})_2$,$(Z_{(0,1),(0,3)})_1,(Z_{(0,1),(0,3)})_2,$ &$i,i,-i,-i$\\
\hline 
$\chi_{24}^4$&$F(W_1),F(W_3)$, $F(W_{21}),F(W_{23})$, &$e^\frac{2\pi i}{3},e^\frac{-\pi i}{3}$,$e^\frac{11\pi i}{12},e^\frac{11\pi i}{12}$\\
\hline 
$\chi_6^2$& $(W_{25})_1,(W_{25})_2,(W_{26})_1,(W_{26})_2$& $e^\frac{\pi i}{4},e^\frac{\pi i}{4},e^\frac{\pi i}{4},e^\frac{\pi i}{4}$\\
\hline
\end{tabular}
\end{center}
In order to show the fusion category $\cA$ categorifies the fusion ring $R'(2)$,  we need to know more information (mainly the $S$-matrix) about modular fusion category $\cD$. 
 The $S$-matrix of fusion subcategory $\cD_0$ is given  in \cite{rowellsolomonzhang},  we determine the $S$-matrix of $\cD$ later. 
\begin{lem}
The simple objects $F(Z_{(0,0),(1,0)})$ and $F(Z_{(0,1),(1,3)})$ are self-dual. Moreover, 
\begin{align*}
 S_{F(Z_{(0,0),(1,0)}),(Z_{(1,1),(1,3)})_i}&=-S_{F(Z_{(0,1),(1,3)}),(Z_{(1,1),(1,3)})_j}=-2\chi_6^3,\\
 S_{F(Z_{(0,0),(1,0)}),(W_k)_i}&=S_{F(Z_{(0,1),(1,3)}),(W_k)_i}=0, \end{align*}
 $k\in\{25,26\}$, $i,j\in\{1,2\}$.
\end{lem}
\begin{proof}
On the contrary, assume $F(Z_{(0,0),(1,0)})^*=F(Z_{(0,1),(1,3)})$, then for any simple object $Z$ of $\cD$, we have $\overline{S_{F(Z_{(0,0),(1,0)}),Z}}=S_{F(Z_{(0,1),(1,3)}),Z}$. However, it follows from \cite[Theorem 2.2]{rowellsolomonzhang} that \begin{align*}
S_{F(Z_{(0,0),(1,0)}),(Z_{(1,1),(1,3)})_1}+S_{F(Z_{(0,0),(1,0)}),(Z_{(1,1),(1,3)})_2}&=-4\chi_6^3,\\
S_{F(Z_{(0,1),(1,3)}),(Z_{(1,1),(1,3)})_1}+S_{F(Z_{(0,1),(1,3)}),(Z_{(1,1),(1,3)})_2}&=4\chi_6^3,
\end{align*}
it is a contradiction. So $F(Z_{(0,0),(1,0)})$ and $F(Z_{(0,1),(1,3)})$ are self-dual. Let $V\in\{F(Z_{(0,0),(1,0)}),F(Z_{(0,1),(1,3)})\}$. Then for  $i\in\{1,3,21,23\}$ and $V'\in\{F(Z_{(0,0),(1,0)}),F(Z_{(0,1),(1,3)})\}$,  \cite{rowellsolomonzhang} implies
\begin{align*}
S_{F(Y_{(0,0)}),V}=2\chi_6^3=-S_{F(Y_{(1,0)}),V}, \quad S_{V,V'}=S_{F(W_i),V}=0.
\end{align*}
 Since $F(X_{(1,0)})$ generates $\text{sVec}$, $S_{F(X_{(1,0)})\otimes W,V}=-S_{W,V}$ for all $W\in\cD_0$ and  $V\notin \cD_0$ by \cite{bruillardetal2017fermion}. 
  Then the unitarity of the normalized $S$-matrix implies   
 $(S_{F(Z_{(0,0),(1,0)}),(Z_{(1,1),(1,3)})_1})^2+(S_{F(Z_{(0,0),(1,0)}),(Z_{(1,1),(1,3)})_2})^2\leq 8(\chi_6^3)^2$,
 which  shows $
 S_{F(Z_{(0,0),(1,0)}),(Z_{(1,1),(1,3)})_i}=-S_{F(Z_{(0,1),(1,3)}),(Z_{(1,1),(1,3)})_j}=-2\chi_6^3$, 
 where $i,j\in\{1,2\}$. Hence, $S_{F(Z_{(0,0),(1,0)}),(W_k)_i}=S_{F(Z_{(0,1),(1,3)}),(W_k)_i}=0$, $k\in\{25,26\}$, $i\in\{1,2\}$.
\end{proof}
 
\begin{lem}
    For simple objects $F(W_{1})$ and $F(W_{3})$, we have $S_{F(W_3),(W_k)_i}=-S_{F(W_1),(W_k)_j}=2\chi_6^2$, where  $k\in\{25,26\}$, $i,j\in\{1,2\}$.
\end{lem}
\begin{proof}
  The object $F(W_3)$ is self-dual, since its twist distinguishes it from all other simple objects of the same dimension. Its $S$-entries are therefore real.   Let $a:=\frac{S_{F(W_3),(W_{25})_1}}{2\chi^2_6}$ and $b:=\frac{S_{F(W_3),(W_{26})_1}}{2\chi^2_6}$. Since  $S_{W_3,W_{25}}=S_{W_3,W_{26}}=4\chi_6^2$ by \cite[Theorem 2.2]{rowellsolomonzhang}, we have
\begin{align*}
\dim(\cD)=\sum_X S_{X,F(W_3)}^2=8(2\chi_6^2)^2+(a2\chi^2_6)^2+((2-a)2\chi^2_6)^2+(b2\chi^2_6)^2+((2-b)2\chi^2_6)^2=48(\chi^2_6)^2,
\end{align*}
which means $a^2+(2-a)^2+b^2+(2-b)^2=4$, so $a=b=1$. Hence $S_{F(W_3),(W_{25})_i}=S_{F(W_3),(W_{26})_j}=2\chi^2_6,$ where $ i,j\in \{1,2\}$.  
\end{proof}

\begin{prop}\label{W21W23}
  For simple objects $F(W_{21})$ and $F(W_{23})$ and $j\in\{1,2\}$, we have   \begin{align*}S_{F(W_{21}),(Z_{(1,1),(1,3)})_j}=0, S_{F(W_{21}),(W_{25})_j}=-S_{F(W_{21}),(W_{26})_j}=2\chi_6^2i.
  \end{align*}
\end{prop}
\begin{proof}
It follows  from \cite[Theorem 2.2]{rowellsolomonzhang} that $
S_{F(W_{21}),F(W_{21})}=-2\chi_6^2i$ and $
S_{F(W_{21}),F(W_{23})}=2\chi_6^2i$, so $F(W_{21})$ is not self-dual; its dimension and twist force $F(W_{21})^*=F(W_{23})$. The same condensation formula gives
\begin{align*}
&S_{F(W_{21}),(W_{25})_1}+S_{F(W_{21}),(W_{25})_2}=4\chi_6^2i,\quad 
S_{F(W_{21}),(W_{26})_1}+S_{F(W_{21}),(W_{26})_2}=-4\chi_6^2i,\\
&S_{F(W_{21}),(Z_{(1,1),(1,3)})_1}+S_{F(W_{21}),(Z_{(1,1),(1,3)})_2}=0.
\end{align*}
Let $a,b_1,b_2\in\bbC$ be defined by
\[
a:=\frac{S_{F(W_{21}),(Z_{(1,1),(1,3)})_1}}{2\chi_6^2},\qquad
b_1:=\frac{S_{F(W_{21}),(W_{25})_1}}{2\chi_6^2},\qquad
b_2:=\frac{S_{F(W_{21}),(W_{26})_1}}{2\chi_6^2}.
\]
After division by $(2\chi_6^2)^2$, the known entries contribute $8$ to the row norm, whose total is $12$. The fermion partners of the two $Z$-entries contribute another $2|a|^2$. Hence
\[
4|a|^2+|b_1|^2+|2i-b_1|^2+|b_2|^2+|-2i-b_2|^2=4.
\]
Equivalently, $4|a|^2+2|b_1-i|^2+2|b_2+i|^2=0$. Thus $a=0$, $b_1=i$ and $b_2=-i$, as required.
\end{proof}

\begin{thm}\label{S-matrix}
  The S-matrix and $T$-matrix of modular category $\mathcal{D}$ are
\[
S=\begin{bmatrix}
\hat{S} & \hat{S} & B\\
\hat{S} & \hat{S} & -B \\
B^T & -B^T & C
\end{bmatrix},
 \text{where}\quad 
\hat{S}=
\begin{bmatrix}
1 & \chi^5_{24} & \chi^3_{6} & \chi^3_{6} & 2\chi^2_{6} \\
\chi^5_{24} & 1 & \chi^3_{6} & \chi^3_{6}  & -2\chi^2_{6}\\
\chi^3_{6}  & \chi^3_{6}  & (\sqrt{2}i-1)\chi^3_{6}  & -(\sqrt{2}i+1)\chi^3_{6} & 0\\
\chi^3_{6}  & \chi^3_{6}   & -(\sqrt{2}i+1)\chi^3_{6}   & (\sqrt{2}i-1)\chi^3_{6} & 0\\
2\chi^2_{6}  & -2\chi^2_{6} & 0 & 0 & 2\chi^2_{6}
\end{bmatrix},
\]\\
\[
B^T=\begin{bmatrix}
2\chi^3_{6} & 2\chi^3_{6} & -2\chi^3_{6} & -2\chi^3_{6} & 0\\
2\chi^3_{6} & 2\chi^3_{6} & 2\chi^3_{6} & 2\chi^3_{6} & 0\\
2\chi^2_{6} & -2\chi^2_{6} & 0 & 0  & -2\chi^2_{6}\\
2\chi^2_{6} & -2\chi^2_{6} & 0 & 0 & -2\chi^2_{6}\\
\chi^2_{6} & -\chi^2_{6} & u_1 & \overline{u_1} & 2\chi^2_{6}\\
\chi^2_{6} & -\chi^2_{6} & u_2 & \overline{u_2}& 2\chi^2_{6}\\
\chi^2_{6} & -\chi^2_{6} & \overline{u_1} & u_1 & 2\chi^2_{6}\\
\chi^2_{6} & -\chi^2_{6} & \overline{u_2} & u_2 & 2\chi^2_{6}
\end{bmatrix},
C=\setlength{\arraycolsep}{3pt}
\begin{bmatrix}
0 & 0 & 0& 0 & 0 & 0 &0 & 0\\
0 & 0 & 0& 0 & 0 & 0 &0 & 0\\
0 & 0  & -2\chi^2_{6}i &  2\chi^2_{6}i &  2\chi^2_{6}i & 2\chi^2_{6}i & -2\chi^2_{6}i & -2\chi^2_{6}i\\
0 & 0  &  2\chi^2_{6}i & -2\chi^2_{6}i & -2\chi^2_{6}i & -2\chi^2_{6}i &  2\chi^2_{6}i & 2\chi^2_{6}i\\
0 & 0  &  2\chi^2_{6}i & -2\chi^2_{6}i & u_3 & u_4& \overline{u_3}& \overline{u_4}\\
0 & 0  &  2\chi^2_{6}i & -2\chi^2_{6}i &  u_4& u_3&\overline{u_4} & \overline{u_3}\\
0 & 0  & -2\chi^2_{6}i &  2\chi^2_{6}i & \overline{u_3}&  \overline{u_4}&u_3 &u_4\\
0 & 0  & -2\chi^2_{6}i &  2\chi^2_{6}i & \overline{u_4}& \overline{u_3} &u_4 & u_3
\end{bmatrix},
\]
\[
\begin{aligned}
T = \text{diag}\;(1,\;1,\;i,\;i,\;e^{-\frac{\pi i}{3}},\;-1,\;-1,\;-i,\;-i,\;e^{\frac{2\pi i}{3}},
1,\;-i,\;e^{\frac{11\pi i}{12}},\;e^{\frac{11\pi i}{12}},\;e^{\frac{\pi i}{4}},\;e^{\frac{\pi i}{4}},\;e^{\frac{\pi i}{4}},\;e^{\frac{\pi i}{4}}),
\end{aligned}
\]
and 
 $u_1=\pm (3\sqrt{2}+2\sqrt{3})i$,
$u_3=(-\sqrt{3}+i)\chi_6^2$, $u_2=-u_1, u_4=-\overline{u_3}$.
\end{thm}

\begin{proof} By Proposition \ref{W21W23}, we know that these simple objects $(W_{k})_j$ are not self-dual and $((W_{k})_j)^*\neq(W_k)_i$ for $k\in\{25,26\}$ and $i,j\in\{1,2\}$ . Without loss of generality, let $(W_{25})^*_1=(W_{26})_1$, $(W_{25})^*_2=(W_{26})_2$.

Let $u_1:=S_{(Z_{(1,1),(1,3)})_1,(W_{25})_1}$ and $u_2:=S_{(Z_{(1,1),(1,3)})_1,(W_{25})_2}$. Applying orthogonality conditions, we have the following equations:
\begin{itemize}
\item From $S_{(Z_{(1,1),(1,3)})_1,-}$  and $ S_{(Z_{(1,1),(1,3)})_2,-}$: $|u_1|^2+|u_2|^2=12(5+2\sqrt{6})$;
\item From $S_{(Z_{(1,1),(1,3)})_1,-}$ and $S_{F(W_{3}),-}$:  $u_1+u_2=-\overline{u_1+u_2}$;
\item From   $S_{(Z_{(1,1),(1,3)})_1,-}$ and $S_{F(W_{21}),-}$: $u_1+u_2=\overline{u_1+u_2}$;
\item From the  $S_{(Z_{(1,1),(1,3)})_1,-}$ and  $S_{(Z_{(1,1),(1,3)})_1,-}$: $2u_1^2+2\overline{u_1}^2+8(\chi^3_6)^2=0$;
\end{itemize}
Thus, we obtain $u_1=\pm(3\sqrt{2}+2\sqrt{3})i$.

Let $u_3:=S_{(W_{25})_1,(W_{25})_1}$, $u_4:=S_{(W_{25})_1,(W_{25})_2}$, $u_5:=S_{(W_{25})_2,(W_{25})_2}$. Again we apply orthogonality conditions, we have:
\begin{itemize}
    \item From  $S_{F(W_3),-}$ and $S_{(W_{25})_1,-}$: $u_3+u_4=-\overline{u_3+u_4}$;
    \item From   $S_{F(W_3),-}$ and $S_{(W_{25})_2,-}$: $u_4+u_5=-\overline{u_4+u_5}$;
    \item From $S_{(W_{25})_1,-}$ and $S_{(W_{26})_1,-}$,  $S_{(W_{25})_2,-}$ and $S_{(W_{26})_2,-}$:
$|u_3|^2+|u_4|^2=8(\chi^2_6)^2=|u_4|^2+|u_5|^2$;
\item From $S_{(W_{25})_1,-}$ and $S_{(W_{25})_2,-}$: $
16(\chi^2_6)^2+u_4(u_3+u_5)+\overline{u_4}( \overline{u_3} + \overline{u_5})=0;
$

\item From $S_{(W_{25})_1,-}$ and $S_{(W_{25})_1,-}$: 
$u_3^2+u_4^2+\overline{u_3}^2+\overline{u_4}^2=8(\chi^2_6)^2$;
\item From $S_{(W_{25})_2,-}$ and $S_{(W_{25})_2,-}$:
$u_4^2+u_5^2+\overline{u_4}^2+\overline{u_5}^2=8(\chi^2_6)^2$;
\end{itemize}
Assume $u_3=a+bi$, then $u_4=-a+ci$ and $u_5=a+di$ where $a,b,c,d\in\bbR$, then the  above equations implies 
\begin{align*}
2a^2+(b+d)c=8(\chi^2_6)^2,\;
4a^2-2b^2-2c^2=8(\chi^2_6)^2,\;
2a^2+b^2+c^2=8(\chi^2_6)^2.
\end{align*}
The last two equations give $a^2=3(\chi_6^2)^2$ and $b^2+c^2=2(\chi_6^2)^2$. Orthogonality with the row $S_{F(W_{21}),-}$ also gives
\[
b+c=c+d=2\chi_6^2.
\]
Hence $b=c=d=\chi_6^2$, so $u_5=u_3$ and $u_4=-\overline{u_3}$. Notice that $\cD\boxtimes\cC(\bbZ/2\bbZ,q)\cong\cZ(\cA)\cong\cZ(\cC)^0_{\bbZ/2\bbZ}$, the central charge of $\xi(\cD)$ is the reverse of $\xi(\cC(\bbZ/2\bbZ,q))=e^{\pi i/4}$ as the central charges of Drinfeld centers are trivial, so 
\[
\tau^+(\cD):=\sum_{V\in\cO(\cD)}\theta_V\dim(V)^2=\sqrt{\dim(\cD)}\xi(\cD)=\sqrt{\dim(\cD)}\,e^{-\pi i/4}.
\]
The $((W_{25})_1,(W_{25})_1)$ entry of $
STS=\tau^+(\cD)T^{-1}ST^{-1}$ 
then gives $u_3=(-\sqrt3+i)\chi_6^2$, fixing $a=-\sqrt3\chi_6^2$. Direct substitution verifies that the unnormalized matrices satisfy equations
$
S\overline S^{\,T}=\dim(\cD)I$ and 
$(ST)^3=\tau^+(\cD)S^2$. 
\end{proof}
\begin{cor}\label{nonselfdual-categorification}
  Let $\cC$ be a near-group fusion category of type $\bbZ/2\bbZ\times\bbZ/4\bbZ+8$, then $\cC\cong\cA^{\bbZ/2\bbZ}$ with $\cK_0(\cA)=R'(2)$.   
\end{cor}
\begin{proof}
Let $\cF:\cZ(\cA)\to\cA$ be the forgetful functor and $\cI$ being the induction functor. Use the order $\cC(\bbZ/2\bbZ,q)\boxtimes\cD$ for the two factors of $\cZ(\cA)$. Since $\cK_0(\cA)$ is commutative, \cite[Theorem 2.13]{ostrik2015pivotal} implies that $\cI(\unit_\cA)$ has six distinct simple summands. Its twist is trivial, and there are exactly six twist-one simple objects, so
\begin{align*}
\cI(\unit_\cA)=\unit\boxtimes\unit\oplus \unit\boxtimes F(Y_{(0,0)})
\oplus \unit\boxtimes F(Z_{(0,0),(1,0)})\oplus b\boxtimes F(Z_{(0,1),(1,3)})
\oplus b\boxtimes (Z_{(0,1),(0,3)})_1\oplus b\boxtimes (Z_{(0,1),(0,3)})_2,
\end{align*}
where $b$ is the unique non-trivial simple object of $\cC(\bbZ/2\bbZ,q)$. Denote by $f:=F(X_{(1,0)})$. The twists of $b\boxtimes\unit$, $\unit\boxtimes f$ and $b\boxtimes f$ are $i,-1,-i$. None is a summand of $\cI(\unit_\cA)$, so  we have $\cF(f)\neq\cF(b)$, $\cF(f)\neq\unit_\cA$ and $\cF(b)\neq\unit_\cA$. Hence $\cA$ contains $\text{Vec}_{\bbZ/2\bbZ\times \bbZ/2\bbZ}^\omega$ as fusion subcategory for some $3$-cocycle $\omega$. Since  the dimensions of simple objects of $\cA$ are either $1$ or $\chi_6^2$, $\cF((W_{25})_1)$ must be a simple object. Choose the labels so that $u_1=(3\sqrt2+2\sqrt3)i$. By using the $S$-matrix obtained in Theorem \ref{S-matrix},  the Verlinde formula (\ref{verlinde}) implies
\begin{align*}
 (W_{25})_1 \otimes (W_{25})_1 = (Z_{(0,1),(0,3)})_1 \oplus F(W_3) \oplus (Z_{(1,1),(1,3)})_2,    
\end{align*}
Notice that $\cF((Z_{(0,1),(0,3)})_j)$ has dimension $1+\chi_6^2$, $j=1,2$, so it contains exactly one invertible simple object. By adjunction and the displayed decomposition of $\cI(\unit_\cA)$, that invertible object must be $\cF(b)$. Since
\[
(Z_{(1,1),(1,3)})_2=f\otimes(Z_{(0,1),(0,3)})_2,
\]
the unique invertible object in its forgetful image is $\cF(b\boxtimes f)$, distinct from $\cF(b)$ and from the unit. On the other hand, $\dim(\cF(F(W_3)))=2\chi_6^2$ shows that this image contains no invertible summands. Thus $\cF((W_{25})_1)\otimes\cF((W_{25})_1)$ contains exactly two distinct non-trivial invertible objects, which implies $\cF((W_{25})_1)$ is not self-dual. As  the four invertible objects  act transitively on the two non-invertible simple objects, $\cA$ is a $\bbZ/2\bbZ\times\bbZ/2\bbZ$-quadratic category of rank six with $\cK_0(\cA)=R'(2)$ by comparing the dimension of simple objects.
\end{proof}

 It is easy to see that  the fusion rings $R(2,2)$, $R(4,4)$ and $R'(4)$  satisfy $d$-number test, cyclotomic field test and pseudo-unitary inequality \cite{ostrik2009formal,ostrik2015pivotal}, we don't know whether these fusion rings are categorifiable, however. We suspect that these three fusion rings are related to near-group fusion categories if they are categorifiable,  and we end this section with the following conjecture.
\begin{conjecture} 
Suppose that $\cC$ is a  $\bbZ/2\bbZ\times\bbZ/2\bbZ$-quadratic fusion category of rank six with the fusion rules determined by \ref{selfdual} or \ref{notselfdual}.
Then $(m,n)\in\{(0,0),(1,0),(0,1),(2,0),(0,2),(2,2),(4,4),(4,6),(6,4)\}$ if $\cC$ is self-dual and
$m\in\{0,2,4\}$ if $\cC$ is not self-dual.
 \end{conjecture}
 
\section*{Acknowledgements}
 Z. Yu thanks A. Schopieray for suggesting this topic and communications on modular fusion categories. The authors are grateful to Q. Zhang for helps in computing the $S$-matrix of   Drinfeld center of near-group fusion category of type $\bbZ/2\bbZ \times \bbZ/4\bbZ+8$. The authors thank  the anonymous referee for   numerous suggestions  and comments that helped improve   this paper substantially. Z.Yu is supported by NSFC (no.12571041) and  Qinglan Project of Yangzhou University.

\author{Yue Meng\\ \thanks{Email:\,DX120250050@stu.yzu.edu.cn}
\\{\small School  of Mathematics,  Yangzhou University,
Yangzhou 225002, China}
}\\

\noindent \author{Zhiqiang Yu\\ \thanks{Email:\,zhiqyumath@yzu.edu.cn}
\\{\small School  of Mathematics,  Yangzhou University,
Yangzhou 225002, China}
}

\end{document}